\documentclass[english, usenames,dvipsnames,svgnames,table]{article}
\usepackage{preamble}

\title{The Gray Product of $(\infty, n)$-Categories via Lax Grids}

\begin{document}

\begin{titlepage}
    \maketitle
    \thispagestyle{empty}
    \begin{abstract}
        We introduce a new model for $(\infty,n)$-categories as Segal sheaves on lax grids, which are pasting diagrams of lax cubes. 
        This model allows for a direct construction of the Gray tensor product via Day convolution.      
        We show that this agrees with Campion's construction of the Gray tensor product.

        These results will be applied in future work to equip the higher categories of cobordisms with a Gray-algebra structure given by the cartesian product of manifolds.
    \end{abstract}

    \bigskip
    \bigskip
    
    \begin{figure}[h]
        \centering
        \includegraphics[width=0.6\linewidth]{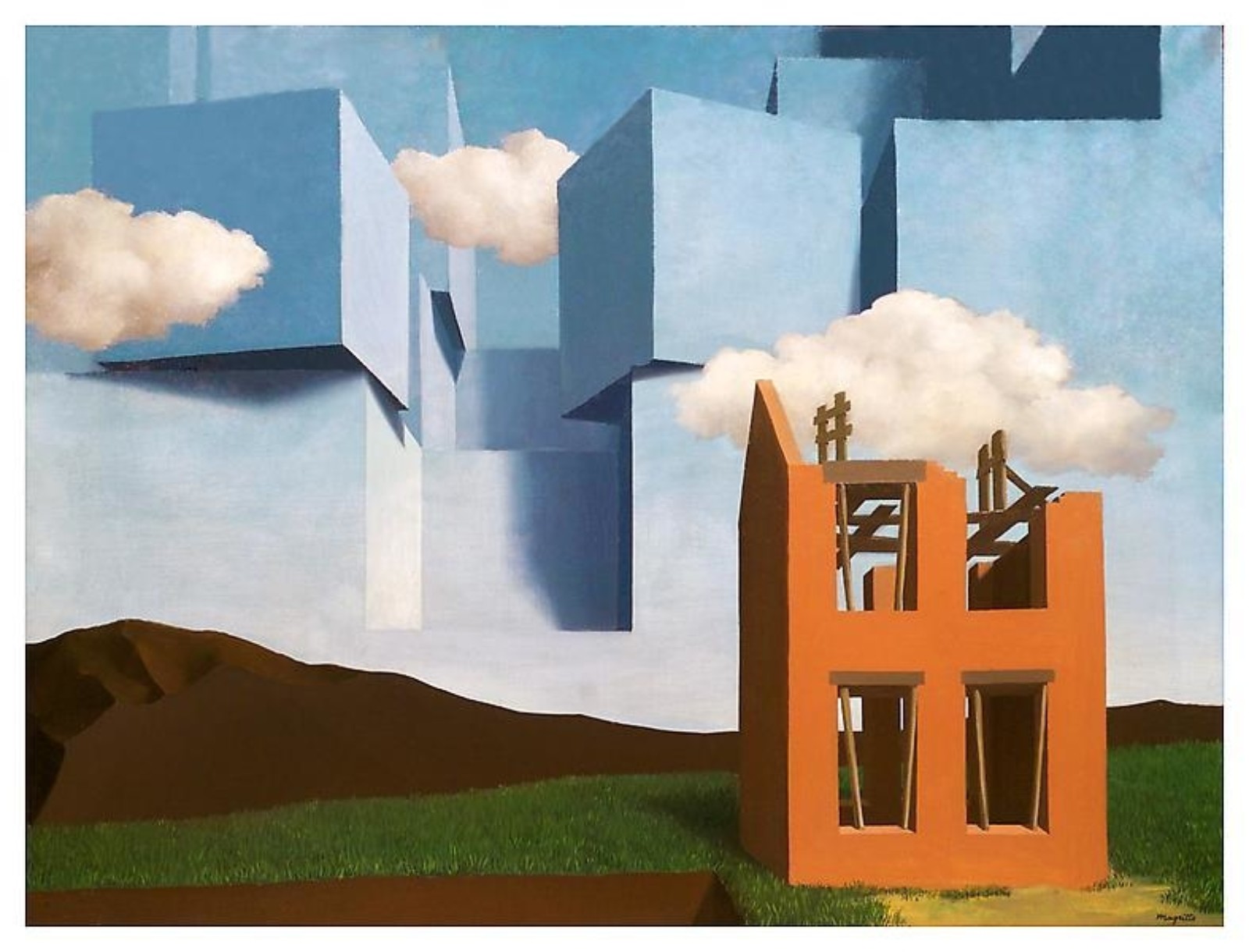}
        \caption*{ Ren\'e Magritte, \textit{L'Univers D\'emasqu\'e (The Universe Unmasked)}, 1932}
    \end{figure}

\end{titlepage}

\tableofcontents
\newpage

\section{Introduction}

Throughout the paper, we refer to $(\infty, n)$-categories as \emph{$n$-categories}, and to $1$-categories also simply as \emph{categories}.
In particular, by an \emph{$\infty$-category} we always mean an $(\infty,\infty)$-category.

\medskip\noindent     
\textbf{The Gray product.}
The Gray product\footnote{Commonly called the Gray \emph{tensor} product.} is a fundamental construction in higher category theory, serving as a lax analog of the cartesian product.
As a central example, the Gray product of two copies of the walking arrow is the \emph{lax square}:
\begin{equation*}
    \begin{tikzcd}
        \bullet & \bullet \\
        \bullet & \bullet.
        \arrow[from=1-1, to=1-2]
        \arrow[from=1-1, to=2-1]
        \arrow[between={0.3}{0.7}, Rightarrow, from=1-2, to=2-1]
        \arrow[from=1-2, to=2-2]
        \arrow[from=2-1, to=2-2]
    \end{tikzcd}
\end{equation*}

The Gray product was originally defined by Gray in~\cite{Gray-1974-Gray} as a closed monoidal structure on strict 2-categories.
Just as the cartesian product of categories induces an internal hom consisting of functors and natural transformations, the Gray product induces an internal hom consisting of functors and \emph{lax natural transformations}.
The latter then serves as a basis for many lax constructions such as lax (co)limits.

Subsequently, the Gray product was extended to strict $\infty$-categories in~\cite{Crans-1995-gray,Al-Agl-Brown-Steiner-2002-strict-globular-cubical,Steiner-2004-gray}, and to non-strict $\infty$-categories in~\cite{Verity-2008-complicial,Johnson-Freyd-Scheimbauer-2017-Funlax,Campion-2023-Gray,Loubaton-2024-infty-categories,Campion-Kapulkin-Maehara-2025-comical,Chanavat-2025-gray};
we refer to the introductions of~\cite{Campion-Maehara-2023-Gray,Loubaton-Ruit-2025-squares} for further review of the history and comparisons of the different models. 

We denote the Gray product of $\infty$-categories by 
\begin{equation*}
    \xlax \;\colon\; \Catoo \times \Catoo \;\too\; \Catoo.
\end{equation*}
In this unbounded setting, the Gray product is additive on categorical dimension: the Gray product of an $n$-category and a $k$-category (viewed as $\infty$-categories) is an $(n+k)$-category.
This induces a functor
\begin{equation*}
    \xlax \;\colon\; \Catn \times \Catn[k] \;\too\; \Catn[(n+k)].
\end{equation*}
As an important example, consider the \emph{lax $n$-cube}, roughly described as an $n$-dimensional cube in which all cells are non-invertible morphisms (see \cite{Al-Agl-Brown-Steiner-2002-strict-globular-cubical} for a precise definition).
The Gray product of the lax $n$-cube and the lax $k$-cube is the lax $(n+k)$-cube.

In~\cite{Campion-2023-Gray}, Campion gave a unique characterization of the Gray product on $\infty$-categories.
Denote by $\cubeCat$ the full subcategory of $\infty$-categories spanned by the lax $n$-cubes for $n<\infty$.
This subcategory is closed under the Gray product.

\begin{theorem}[{\cite[Theorem A]{Campion-2023-Gray}}]\label{thm:campion-characterization}
    There is a unique closed monoidal structure on $\infty$-categories that restricts to the Gray product on $\cubeCat$.
\end{theorem}

The proof of \cref{thm:campion-characterization} relies on a previous result of Campion, that $\cubeCat$ is dense in $\Catoo$~\cite{Campion-2022-cubes-dense}.
Namely, every $\infty$-category is canonically a colimit of lax cubes, and the Gray product commutes with colimits in each variable, so computing it reduces to the case of lax cubes.
In principle, this lets us compute any Gray product, provided we can write the $\infty$-categories involved explicitly in terms of lax cubes, which raises the following question:
\begin{question}
\label{que:constructing-infty-categories-with-cubes}
    How does one explicitly construct an $\infty$-category from lax cubes?
\end{question}
A natural approach uses the restricted Yoneda functor along $\cubeCat\subseteq\Catoo$.
Since $\cubeCat$ is dense, this is a fully faithful embedding $\Catoo\into\PSh(\cubeCat)$, and answering \cref{que:constructing-infty-categories-with-cubes} amounts to constructing a presheaf on $\cubeCat$ that lies in its essential image.

This essential image was recently characterized by Gepner--Heine~\cite[Corollary~1.5.5]{Gepner-Heine-2026-Street}, resolving an open question of Campion.
Their characterization uses the fact that cubes are glued together from globes, i.e.\ walking $n$-morphisms.
For our purposes, most notably the construction of cobordism categories, a description purely in terms of cubes is preferable.

In this paper, we give a different answer to \cref{que:constructing-infty-categories-with-cubes}, based on \emph{lax grids}: pasting diagrams of lax cubes.
We model $\infty$-categories as presheaves on lax grids that satisfy the Segal condition with respect to these pasting diagrams.
Restricting to lax grids of categorical dimension $\le n$ gives a model for $n$-categories.
We first review the case $n=1$, where our model reduces to Rezk's \emph{complete Segal spaces}~\cite{Rezk-2001-Segal-spaces}.

\medskip\noindent     
\textbf{Complete Segal spaces.}
A 1-category consists of spaces of objects and morphisms, together with a coherently associative and unital composition.
This structure can be presented as a presheaf on the simplex category~$\deltaCat$.
The singleton $[0]$ represents objects, $[1]$ represents morphisms, and $[k]$ represents strings of $k$ composable morphisms.
The maps in $\deltaCat$ then coherently encode composition.
Formally, there is a fully faithful embedding $\Catn[1]\into \PSh(\deltaCat)$ defined by the restricted Yoneda functor along $\deltaCat\subseteq \Catn[1]$.
Moreover, we can explicitly describe its essential image.

A \emph{Segal space} is a presheaf $X\in \PSh(\deltaCat)$ satisfying the \emph{Segal condition}: the canonical map
\begin{equation*}
    X_k\;\too\; X_1\times_{X_0}\dots \times_{X_0} X_1
\end{equation*}
induced by selecting the $k$ segments of $[k]$ is an isomorphism.
A Segal space is \emph{univalent} (a.k.a.\ \emph{complete}) if paths in the space of objects are identified with invertible morphisms.
We denote by
\begin{equation*}
    \Shv(\deltaCat)\univ \;\subseteq\; \Shv(\deltaCat) \;\subseteq\; \PSh(\deltaCat)
\end{equation*}
the full subcategories of univalent Segal spaces and Segal spaces, respectively.

The essential image of the embedding $\Catn[1]\into \PSh(\deltaCat)$ is precisely the univalent Segal spaces, giving an equivalence
\begin{equation*}
    \Catn[1]\;\simeq\; \Shv(\deltaCat)\univ.
\end{equation*}

\subsection{Main result}
An $n$-category consists of spaces of $k$-morphisms for $0\le k\le n$, together with coherent compositions.
In our model, the $k$-morphisms are represented by lax $k$-cubes, and their compositions by lax grids, which are pasting diagrams of such cubes.

The category of lax grids contains the lax cubes, and like them is closed under the Gray product, thus similarly characterizing it.
Furthermore, lax grids allow for a direct identification of $n$-categories among their presheaves, generalizing the case $n=1$.

\begin{definition}
    The category of lax grids $\grid[]$ is the full subcategory of $\infty$-categories generated by $\deltaCat$ under the Gray product.
    We also write $\grid \subseteq \grid[]$ for the full subcategory of lax grids of categorical dimension $\le n$.
\end{definition}

Lax grids were constructed as strict $\infty$-categories by a combinatorial model in~\cite{Al-Agl-Brown-Steiner-2002-strict-globular-cubical}.
For example, the following diagrams describe objects in $\grid[2]$:
\begin{equation*}
    \begin{tikzcd}
        {} & { } & { } & \bullet & \bullet \\
        {} & {} & { } & \bullet & \bullet
        \arrow["{\large\bullet}"{marking, allow upside down}, draw=none, from=1-1, to=2-1]
        \arrow[""{name=0, anchor=center, inner sep=0}, "{{\large{\bullet}}}"{marking, allow upside down}, draw=none, from=1-2, to=2-2]
        \arrow[""{name=1, anchor=center, inner sep=0}, "{{\large{\bullet}}}"{marking, allow upside down}, draw=none, from=1-3, to=2-3]
        \arrow[from=1-4, to=1-5]
        \arrow[from=1-4, to=2-4]
        \arrow[between={0.3}{0.7}, Rightarrow, from=1-5, to=2-4]
        \arrow[from=1-5, to=2-5]
        \arrow[from=2-4, to=2-5]
        \arrow[between={0.2}{0.8}, from=0, to=1]
    \end{tikzcd}
\end{equation*}
\begin{equation*}
    \begin{tikzcd}
        \bullet & \bullet & \bullet & \bullet & \bullet & \bullet & \bullet & \bullet \\
        \bullet & \bullet & \bullet & \bullet & \bullet & \bullet & \bullet & \bullet \\
        &&& \bullet & \bullet & \bullet & \bullet & {\bullet.}
        \arrow[from=1-1, to=1-2]
        \arrow[from=1-1, to=2-1]
        \arrow[from=1-2, to=1-3]
        \arrow[between={0.3}{0.7}, Rightarrow, from=1-2, to=2-1]
        \arrow[from=1-2, to=2-2]
        \arrow[between={0.3}{0.7}, Rightarrow, from=1-3, to=2-2]
        \arrow[from=1-3, to=2-3]
        \arrow[from=1-4, to=1-5]
        \arrow[from=1-4, to=2-4]
        \arrow[between={0.3}{0.7}, Rightarrow, from=1-5, to=2-4]
        \arrow[from=1-5, to=2-5]
        \arrow[from=1-6, to=1-7]
        \arrow[from=1-6, to=2-6]
        \arrow[from=1-7, to=1-8]
        \arrow[between={0.3}{0.7}, Rightarrow, from=1-7, to=2-6]
        \arrow[from=1-7, to=2-7]
        \arrow[between={0.3}{0.7}, Rightarrow, from=1-8, to=2-7]
        \arrow[from=1-8, to=2-8]
        \arrow[from=2-1, to=2-2]
        \arrow[from=2-2, to=2-3]
        \arrow[from=2-4, to=2-5]
        \arrow[from=2-4, to=3-4]
        \arrow[between={0.3}{0.7}, Rightarrow, from=2-5, to=3-4]
        \arrow[from=2-5, to=3-5]
        \arrow[from=2-6, to=2-7]
        \arrow[from=2-6, to=3-6]
        \arrow[from=2-7, to=2-8]
        \arrow[between={0.3}{0.7}, Rightarrow, from=2-7, to=3-6]
        \arrow[from=2-7, to=3-7]
        \arrow[between={0.3}{0.7}, Rightarrow, from=2-8, to=3-7]
        \arrow[from=2-8, to=3-8]
        \arrow[from=3-4, to=3-5]
        \arrow[from=3-6, to=3-7]
        \arrow[from=3-7, to=3-8]
    \end{tikzcd}
\end{equation*}

In general, objects of $\grid$ are pasting diagrams of lax cubes, as in the examples above.
A \emph{Segal sheaf} on $\grid$ is a presheaf satisfying the Segal condition determined by these pasting diagrams (see~\cref{subsec:grid-spaces} for the precise definition).

\begin{definition}
    Denote by  
    \begin{equation*}
        \Shv(\grid)\univ\;\subseteq\; \Shv(\grid)\;\subseteq\; \PSh(\grid)
    \end{equation*}
    the reflective subcategories of univalent Segal sheaves and Segal sheaves, respectively.
\end{definition}

The category $\grid[]$ is closed under the Gray product.
The resulting monoidal structure induces a Day convolution monoidal structure on $\PSh(\grid[])$, which in turn induces a monoidal structure on $\Shv(\grid[])\univ$ by reflection.

Our main theorem identifies univalent Segal sheaves on $\grid[]$, equipped with the above monoidal structure, with $\infty$-categories, equipped with the Gray product characterized by~\cref{thm:campion-characterization}.

\begin{alphThm}[{\cref{thm:equivalence-infty}, \cref{cor:equivalence-univalence}}]
\label{alphthm:equivalence}
    There is an equivalence of monoidal categories
    \begin{equation*}
        \Catoo \;\simeq\; \Shv(\grid[])\univ.
    \end{equation*}
    Moreover, for $n<\infty$ there is an equivalence of categories
    \begin{equation*}
        \Catn \;\simeq\; \Shv(\grid)\univ.
    \end{equation*}
\end{alphThm}

The latter equivalence lets us describe the Gray product of $n$- and $k$-categories as the functor
\begin{equation*}
    \xlax\;\colon\; \Shv(\grid[n])\univ\;\times\; \Shv(\grid[k])\univ\;\too\; \Shv(\grid[n+k])\univ
\end{equation*}
induced by the Gray product on lax grids $\grid[n]\times\grid[k]\to\grid[n+k]$.

Our results extend to the non-univalent setting. Following \cite{Ayala-Francis-2018-flagged-categories}, we refer to non-univalent \mbox{$n$-categories} as \emph{flagged} $n$-categories, and denote their category by $\Catnfl$. 
Dropping univalence on both sides, we obtain
\begin{equation*}
    \Catoofl \;\simeq\; \Shv(\grid[]),\qquad \Catnfl \;\simeq\; \Shv(\grid).
\end{equation*}

\begin{remark}
    Our results and methods are inspired by~\cite{Al-Agl-Brown-Steiner-2002-strict-globular-cubical}.
    There, Al-Agl--Brown--Steiner introduce a model for strict $\infty$-categories, the \emph{strict cubical $\infty$-categories with connections}.
    This model carries a natural monoidal structure, which they use to define the strict Gray product.
    Our model generalizes theirs: a set-valued $\grid[]$-sheaf is precisely a strict cubical $\infty$-category with connections.
\end{remark}

\medskip\noindent
\textbf{Cobordism categories.}
A primary example motivating our work is that of cobordism categories, which are higher categories arising from geometry.
Denote by $\Bordn$ the $n$-category of fully extended cobordisms, defined originally in~\cite{Lurie-2009-cobordism,Calaque-Scheimbauer-2019-cobordism-category}.
Its objects are $0$-dimensional manifolds, its $1$-morphisms are cobordisms between them, and its higher morphisms are iterated cobordisms, given by manifolds with higher corners.

The disjoint union of manifolds equips $\Bordn$ with a symmetric monoidal structure.
In contrast, the cartesian product of manifolds currently lacks a categorical interpretation.

In future work, we realize $\Bordn$ as a Segal sheaf on $\grid$, where the space of lax $i$-cubes is given by \emph{$i$-cubical cobordisms}, which are certain $i$-manifolds with higher corners (see \cite{Calaque-Scheimbauer-2019-cobordism-category}).
Using our model of the Gray product, we then interpret the cartesian product of manifolds categorically, as a functor
\begin{equation*}
    \Bordn[n]\;\xlax\; \Bordn[k]\;\too\; \Bordn[(n+k)]\qin \Catn[(n+k)]
\end{equation*}
sending a pair of cubical cobordisms to their cartesian product.
This functor lifts to an algebra structure on $\Bordn[\infty]\in \Catoo$ with respect to the Gray product.
A similar construction applies to embedded cobordisms and to cobordisms with tangential structure.

\subsection{Proof idea}

The equivalence in~\cref{alphthm:equivalence} is given by the \emph{cubical nerve}, which is the restricted Yoneda functor along $\grid \subseteq \Catn$:
\begin{equation*}
    \cubicalNerve \;\colon\; \Catn \;\too\; \Shv(\grid)\univ.
\end{equation*}
It suffices to prove that the cubical nerve is an equivalence for $n<\infty$; the infinite case then follows by passing to colimits.
This equivalence is moreover monoidal by the uniqueness in~\cref{thm:campion-characterization}, since it identifies the lax cubes in $\Catoo$ with the representable cubes in $\Shv(\grid[])\univ$.

To prove the finite case, we construct an explicit inverse
\begin{equation*}
    \Fold \;\colon\; \Shv(\grid)\univ \;\too\; \Catn,
\end{equation*}
employing \emph{$n$-uple categories}.

\medskip\noindent     
\textbf{$n$-Uple categories.}
An $n$-uple category is a higher categorical structure in which there are $n$ different types of $1$-morphisms, and higher cells are cubes in which parallel edges have the same type.

As an example, for $n=2$ these are \emph{double categories}.
A double category consists of a space of objects, two spaces of morphisms called \emph{horizontal} and \emph{vertical}, and a space of 2-cells given by squares with horizontal and vertical edges:
\begin{equation*}
    \begin{tikzcd}
    	{} & {} & \bullet & \bullet & \bullet \\
    	{} & {} & \bullet & \bullet & \bullet.
    	\arrow[""{name=0, anchor=center, inner sep=0}, "{{\large\bullet}}"{description}, draw=none, from=1-1, to=2-1]
    	\arrow[""{name=1, anchor=center, inner sep=0}, "{{\large\bullet}}"{description}, draw=none, from=1-2, to=2-2]
    	\arrow["v"', color={rgb,255:red,214;green,92;blue,92}, from=1-3, to=2-3]
    	\arrow[""{name=2, anchor=center, inner sep=0}, "{{h_0}}", color={rgb,255:red,92;green,92;blue,214}, from=1-4, to=1-5]
    	\arrow["{{v_0}}"', color={rgb,255:red,214;green,92;blue,92}, from=1-4, to=2-4]
    	\arrow["{{v_1}}", color={rgb,255:red,214;green,92;blue,92}, from=1-5, to=2-5]
    	\arrow[""{name=3, anchor=center, inner sep=0}, "{{h_1}}"', color={rgb,255:red,92;green,92;blue,214}, from=2-4, to=2-5]
    	\arrow["h", color={rgb,255:red,92;green,92;blue,214}, between={0.2}{0.8}, from=0, to=1]
    	\arrow["\square"{description}, draw=none, from=2, to=3]
    \end{tikzcd}
\end{equation*}

We model $n$-uple categories as \emph{$n$-uple univalent Segal spaces} \cite{Haugseng-2017-Morita}, which are univalent\footnote{It is common to impose univalence in only one of the $n$ directions; here we impose it in all directions.} Segal sheaves on $\deltaCat^n$.
We denote the category of $n$-uple univalent Segal spaces by
\begin{equation*}
    \Catnuple \;\coloneqq\; \Shv(\deltaCat^n)\univ \;\subseteq\; \PSh(\deltaCat^n).
\end{equation*}

\medskip\noindent     
\textbf{From $n$-uple categories to $n$-categories.}
An $n$-uple category is similar to an $n$-category, except that there are multiple types of cells of each dimension instead of a single type.
There are two approaches to obtain $n$-categories from $n$-uple categories:
\begin{enumerate}
    \item require the cells of all types but one to be degenerate;
    \item identify the cells across the different types.
\end{enumerate}
The first approach models $n$-categories as \emph{$n$-fold univalent Segal spaces} \cite{Rezk-2010-n-categories}, which are $n$-uple univalent Segal spaces satisfying a constancy condition. 
For example, when $n=2$, a $2$-fold univalent Segal space is a double category in which all vertical morphisms are degenerate.

Realizing $\Catn$ as the $n$-fold univalent Segal spaces gives an embedding
\begin{equation*}
    \G^*\;\colon\; \Catn\;\intoo\; \Catnuple.
\end{equation*}
This embedding has a right adjoint 
\begin{equation*}
    \G_*\;\colon\; \Catnuple \;\too\; \Catn
\end{equation*}
which discards the non-degenerate cells in all but one of the types.

The second approach is implemented by $\Shv(\grid)\univ$.
Consider the functor
\begin{equation*}
     \square\;\colon\; \deltaCat^n\;\too\; \grid,\qquad ([a_1],\dots,[a_n])\mapsto [a_1]\xlax \cdots \xlax[a_n].
\end{equation*}
Restriction along $\square$ preserves univalent sheaves, giving a functor
\begin{equation*}
     \square^*\;\colon\; \Shv(\grid)\univ\;\too\; \Catnuple.
\end{equation*}
The essential image of this functor consists of $n$-uple categories whose different types of cells are identified.

\medskip\noindent     
\textbf{The equivalence.}
We define $\Fold\colon \Shv(\grid)\univ\to \Catn$ as the composition
\begin{equation*}
    \Shv(\grid)\univ \;\xtoo{\square^*}\; \Catnuple \;\xtoo{\G_*}\; \Catn.
\end{equation*}
Informally, $\Fold$ sends a univalent $\grid$-sheaf to a symmetric $n$-uple category, then discards the cells of all types but one.
This is inverse to the cubical nerve: intuitively, the nerve spreads the cells of a given dimension symmetrically across all types, and $\Fold$ recovers them in a single type.

\begin{remark}
    The composition
    \begin{equation*}
        \Catn \;\xtoo{\cubicalNerve}\; \Shv(\grid)\univ \;\xtoo{\square^*}\; \Catnuple
    \end{equation*}
    is known as the \emph{cubes functor}, or the \emph{squares functor} when $n=2$.
    The squares functor was conjectured to be fully faithful in~\cite{Gaitsgory-Rozenblyum-2019-DAG}, following earlier work of~\cite{Ehresmann-1963-squares,Grandis-Pare-2004-square}.
    This was proven in~\cite{Abellan-2023-squares}, and the essential image of the squares functor was identified in~\cite{Loubaton-2025-squares} (see also~\cite{Loubaton-Ruit-2025-squares}).
    In forthcoming work, we will use the lax grids model to extend these results to all $n$.
\end{remark}

\subsection{Conventions}
    We use the following terminology and notation:
    \begin{itemize}
        \item $\spc$ denotes the category of spaces (a.k.a.\ animae, or $\infty$-groupoids).
        \item $\Cat$ denotes the category of small categories.
        \item $\CAT$ denotes the (huge) category of large categories.
        \item $\Cat_{(1,1)}$ denotes the category of small $(1,1)$-categories.
        \item $\PrL$ and $\PrR$ denote the categories of presentable categories with left and right adjoint functors respectively.
        \item We denote adjunctions by $L \colon \cC \adj \cD \cocolon R$ to mean that $L$ is  left adjoint to $R$.
    \end{itemize}

\subsection{Acknowledgments}

    We are grateful to Tim Campion, Branko Juran, Shaul Ragimov, Jaco Ruit, Tomer Schlank, and Lior Yanovski for many extensive and helpful discussions.
    We thank Shay Ben-Moshe, Avital Binyamin, Beckham Myres, Shaul Ragimov and especially Lior Yanovski for helpful comments on previous drafts.
    Part of this work was carried out during visits to the University of Chicago, whose hospitality we gratefully acknowledge.
    
    We used LLMs to improve the writing; all mathematical content and any errors are our own.

\section{Models for higher categories}
\label{sec:models-for-higher-categories}
We review several well-established models for higher categories, namely $\globe$-spaces and $n$-fold Segal spaces. We also consider the related $n$-uple Segal spaces. All of these models, as well as the lax grids model developed later, are formulated within the framework of \emph{algebraic patterns} introduced by Chu and Haugseng in~\cite{Chu-Haugseng-2021-algebraic-patterns}.

Most of this section is expository; the reader familiar with the material is encouraged to skim or skip ahead and return as needed.

\subsection{Algebraic and geometric patterns}
    Algebraic patterns provide a general framework for describing algebraic and categorical structures in terms of Segal conditions.
    In~\cite{Chu-Haugseng-2021-algebraic-patterns}, \emph{Segal objects} over an algebraic pattern $\cO$ are defined as a localization of presheaves on $\cO\op$.  To avoid cumbersome notation, for this paper we will call $\cX \coloneqq \cO\op$ a \emph{geometric pattern}, and consider $\cO$-Segal objects as sheaves on $\cX$.
    \begin{definition}
        A \emph{geometric pattern} is a category $\cX$ together with
        \begin{enumerate}
            \item a factorization system $(\cX^{\act}, \cX^{\inrt})$ of \emph{active} and \emph{inert} maps, that is, every map in $\cX$ factors uniquely as a composition of an active map followed by an inert map;
            \item a full subcategory $\cX^{\el} \subseteq \cX^{\inrt}$ of \emph{elementary objects}.
        \end{enumerate}
    \end{definition}
    
    \begin{definition}
        A morphism of geometric patterns is a functor $\cX \to \cY$ that preserves inert and active maps and sends elementary objects to elementary objects. 
        We denote the category of geometric patterns by $\GPat$.
    \end{definition}
    The category of geometric patterns is equivalent to the category of algebraic patterns via $\cX \mapsto \cX\op$. See~\cite[\textsection~5]{Chu-Haugseng-2021-algebraic-patterns} for a formal construction of this category.

    Let $\cX$ be a geometric pattern.
    For $X\in \cX$, we denote by $\cX^\el_{/X}$ the full subcategory of $\cX^\inrt_{/X}$ on elementary objects.

    \begin{definition}
        Let $\cC \in \PrL$. A presheaf $F \colon \cX\op \to \cC$ is said to be a \emph{Segal sheaf}\footnote{The name Segal sheaf is common in the literature, but we note that they are usually not sheaves on a site.} if it is a Segal object as in~\cite{Chu-Haugseng-2021-algebraic-patterns}. That is, if for any $X \in \cX$
        \begin{equation*}
            F(X) \;\isotoo\; \lim_{E \in (\cX^{\el}_{/X})\op} F(E).
        \end{equation*}
        We will write $\Shv(\cX; \cC) \subseteq \PSh(\cX;\cC)$ for the full subcategory of Segal sheaves. In the case $\cC=\spc$, we will write $\Shv(\cX) \coloneqq \Shv(\cX; \spc)$.
    \end{definition}

    \begin{lemma}
    \label{lem:Shv(X;C)=Shv(X)@C}
        Let $\cC \in \PrL$. Then $\Shv(\cX; \cC) \simeq \Shv(\cX) \otimes \cC$.
    \end{lemma}
    \begin{proof}
        By~\cite[Lemma~2.11]{Chu-Haugseng-2021-algebraic-patterns}, $\Shv(\cX; \cC)$ is an accessible localization of $\PSh(\cX; \cC) \simeq \PSh(\cX) \otimes \cC$. The result follows as the Lurie tensor product commutes with accessible localizations.
    \end{proof}

    \begin{definition}
        A geometric pattern $\cX$ is called \emph{saturated} if for every $X\in \cX$, the representable presheaf $\Map_{\cX}(-,X)$ is a Segal sheaf, or equivalently, $X$ is a colimit cocone of $\cX^{\el}_{/X}$.
        If $\cX$ is saturated, then the Yoneda embedding factors through sheaves
        \begin{equation*}
            \Yo \;\colon\; \cX \;\intoo\; \Shv(\cX).
        \end{equation*}
        In this case, we usually omit the notation $\Yo$ and identify $\cX \subseteq \Shv(\cX)$.
    \end{definition}
    While not required by definition, all examples of geometric patterns in this work are found to be saturated.        
    
    \begin{definition}
        A \emph{Segal} morphism of geometric patterns $f \colon \cX \to \cY$ is a morphism such that for any $X \in \cX$ and $F \in \Shv(\cY)$, the induced functor $\cX^{\el}_{/X} \to \cY^{\el}_{/f(X)}$ induces an isomorphism
        \begin{equation*}
            F(f(X)) \;\simeq\; \lim_{(\cY^{\el}_{/f(X)})\op} F \;\isotoo\; \lim_{(\cX^{\el}_{/X})\op} F \circ f.
        \end{equation*}
        Equivalently, by~\cite[Lemma~4.5]{Chu-Haugseng-2021-algebraic-patterns}, $f$ is Segal if and only if the restriction 
        \begin{equation*}
            f^*\;\colon\; \PSh(\cY)\;\too\; \PSh(\cX)
        \end{equation*}
        sends Segal sheaves to Segal sheaves.
        We denote the category of geometric patterns with Segal morphisms by $\GPatSeg$.
    \end{definition}

    If $\cX, \cY$ are geometric patterns, then by~\cite[Example~5.7]{Chu-Haugseng-2021-algebraic-patterns}, the category $\cX \times \cY$ admits a structure of a geometric pattern with both active and inert maps, and elementary objects, computed coordinate-wise. 
    Moreover, for any $\cC \in \PrL$
    \begin{equation*}
        \Shv(\cX \times \cY; \cC) \;\simeq\; \Shv(\cX; \Shv(\cY; \cC)).
    \end{equation*}

    \begin{corollary}
    \label{cor:product-of-geometric-patterns}
        The categories $\GPat$ and $\GPatSeg$ have finite products, and for any $\cX,\cY \in \GPat$
        \begin{equation*}
            \Shv(\cX \times \cY) \simeq \Shv(\cX) \otimes \Shv(\cY).
        \end{equation*}
    \end{corollary}
    \begin{proof}
        The equivalence follows from~\cref{lem:Shv(X;C)=Shv(X)@C}.
        It remains to prove that the projection maps $p_X \colon \cX \times \cY \to \cX$ and $p_Y \colon \cX \times \cY \to \cY$ are Segal, or equivalently, that $p_X^*$ and $p_Y^*$ preserve Segal sheaves. This follows, as under the above equivalence, $p_X^*$ and $p_Y^*$ are identified with the compositions
        \begin{equation*}
            \begin{split}
                & \Shv(\cX) \;\simeq\; \Shv(\cX) \otimes \spc \;\too\; \Shv(\cX) \otimes \Shv(\cY) \;\simeq\; \Shv(\cX\times\cY) \\
                & \Shv(\cY) \;\simeq\; \spc \otimes \Shv(\cY) \;\too\; \Shv(\cX) \otimes \Shv(\cY) \;\simeq\; \Shv(\cX\times\cY)
            \end{split}
            \qin \PrR.
        \end{equation*}
    \end{proof}

    \begin{definition}
        We call the equivalence from~\cref{cor:product-of-geometric-patterns} the \emph{external product}
        \begin{equation*}
            \boxtimes \;\colon\; \Shv(\cX)\otimes \Shv(\cY) \;\isotoo\; \Shv(\cX\times \cY).
        \end{equation*}
    \end{definition}

    Let $f \colon \cX \to \cY$ be a Segal morphism of geometric patterns and $\cC \in \PrL$. By~\cite[Proposition~4.7]{Chu-Haugseng-2021-algebraic-patterns}, there is an adjunction
    \begin{equation*}
        f_! \;\colon\; \Shv(\cX; \cC) \;\longadj\; \Shv(\cY; \cC) \;\cocolon\; f^*,
    \end{equation*}
    where $f^*$ is the restriction of presheaves and $f_!$ is given by left Kan extension followed by Segalification.

    \begin{lemma}
    \label{lem:sheavs-are-symmetric-monoidal}
        The Segal sheaves construction gives rise to a symmetric monoidal functor 
        \begin{equation*}
            \Shv \;\colon\; \GPatSeg \;\too\; \PrL,
        \end{equation*}
        with the $(-)_!$-functoriality.
    \end{lemma}
    \begin{proof}
        The functor $\PSh \colon (\GPatSeg)\op \to \PrR$ with the $(-)^*$-functoriality is symmetric monoidal. The subfunctor $\Shv \colon (\GPatSeg)\op \to \PrR$ is therefore symmetric monoidal by~\cref{cor:product-of-geometric-patterns}. The result follows by passing to left adjoints.
    \end{proof}

    \subsubsection*{Colimits of geometric patterns}
 
    By~\cite[Corollary~5.5]{Chu-Haugseng-2021-algebraic-patterns}, $\GPat$ has all filtered colimits, and they are preserved by the forgetful functor $\GPat\to \Cat$. 
    This does not seem to be the case for $\GPatSeg$ (\cite[Remark~5.6]{Chu-Haugseng-2021-algebraic-patterns}). However, certain colimits in $\GPatSeg$ can still be computed by imposing a stronger condition on the morphisms.

    \begin{definition}
        A morphism of geometric patterns $f \colon \cX \to \cY$ is called \emph{strong Segal} if for every $X\in \cX$, the induced functor $\cX^{\el}_{/X} \to \cY^{\el}_{/f(X)}$ is cofinal.
    \end{definition}
    
    \begin{proposition}
    \label{prop:gpat-colimit-segal-sheaves}
        Let $\cX_\bullet \colon I \to \GPatSeg$ be a diagram. Assume that:
        \begin{enumerate}
            \item the colimit $\cX \coloneqq \colim_{i \in I} \cX_i \in \GPat$ exists,
            \item this colimit is preserved by the forgetful functor $\GPat\to \Cat$, and
            \item the structure maps $f_i \colon \cX_i \to \cX$ are strong Segal.
        \end{enumerate} 
        Then $\cX$ is also a colimit of $\cX_\bullet$ in $\GPatSeg$, and this colimit is preserved by $\Shv \colon \GPatSeg \to \PrL$.
    \end{proposition}
    \begin{proof}
        We will first prove that $\Shv(\cX)$ is a colimit of $\Shv(\cX_\bullet)$ in $\PrL$, or equivalently, by passing to right adjoints, that it is the limit in $\CAT$.
        Since $\cX\simeq \colim_{i\in I} \cX_i\in \Cat$, we have
        \begin{equation*}
            \PSh(\cX)\;\isotoo\; \lim_{i\in I\op}\PSh(\cX_i)\qin \CAT.
        \end{equation*}
        It then suffices to show that $F\in \PSh(\cX)$ is Segal if and only if $F|_{\cX_i}$ is Segal for all $i$. The \quotes{only if} direction follows from the fact that each $\cX_i\to \cX$ is a Segal morphism. 
        For the \quotes{if} direction, let $X\in \cX$. Since $\cX$ is the colimit of $\cX_\bullet$ as categories, $X \simeq f_i(X_i)$ for some $i\in I$ and $X_i\in \cX_i$. 
        As $f_i$ is strong Segal, the functor $(\cX_i)^{\el}_{/X_i}\to \cX^{\el}_{/X}$ is cofinal, so
        \begin{equation*}
            \lim_{(\cX^{\el}_{/X})\op} F \;\simeq\; \lim_{((\cX_i)^{\el}_{/X_i})\op} F\circ f_i \;\simeq\; F(f_i(X_i)) \;\simeq\; F(X),
        \end{equation*}
        where the second isomorphism uses that $\restrict{F}{\cX_i}$ is Segal.

        It remains to show that $\cX$ is the colimit of $\cX_\bullet$ in $\GPatSeg$.
        Let $\cY$ be a cocone under $\cX_\bullet$ in $\GPatSeg$.
        Since $\cX$ is a colimit in $\GPat$, we get a morphism of geometric patterns $\cX\to \cY$ which we want to prove is Segal. 
        Notice that a Segal sheaf on $\cY$ restricts to a Segal sheaf on $\cX_i$ for every $i\in I$.
        Thus, from the above paragraph, the restriction $\PSh(\cY)\to \PSh(\cX)$ preserves Segal sheaves.
    \end{proof}
    
    \begin{corollary}
    \label{cor:gpat-filtered-colimit-strong-segal}
        Let $\cX_{\bullet} \colon I \to \GPatSeg$ be a filtered diagram consisting of strong Segal morphisms. Then this diagram admits a colimit in $\GPatSeg$ and
        \begin{equation*}
            \colim_i \Shv(\cX_i) \;\isotoo\; \Shv(\colim_i \cX_i) \qin \PrL.
        \end{equation*}
    \end{corollary}
    \begin{proof}
        Let $\cX \coloneqq \colim_{i \in I} \cX_i$ be the colimit computed in $\GPat$. As filtered colimits in $\GPat$ commute with the forgetful functor to $\Cat$, by~\cref{prop:gpat-colimit-segal-sheaves}, it is enough to prove that the structure maps $f_i \colon \cX_i \to \cX$ are strong Segal. That is, for any $X\in \cX_i$, the functor $(\cX_i)^{\el}_{/X} \to \cX^{\el}_{/f_i(X)}$ is cofinal.
        Since filtered colimits in $\Cat$ commute with overcategories~(\cite[Lemma~9.1.9.7]{kerodon}), and the structure maps preserve inert maps and elementary objects, we have
        \begin{equation*}
            \cX^{\el}_{/f_i(X)} \;\simeq\; \colim_{i\xto{\alpha} j \;\in\; I_{i/}} (\cX_j)^{\el}_{/f_{\alpha}(X)}
        \end{equation*}
        where $f_\alpha\colon \cX_i\to \cX_j$ are the transition functors. 
        Each transition functor in this filtered diagram is cofinal by the strong Segal assumption, so the map from the initial term $(\cX_i)^{\el}_{/X} \to \cX^{\el}_{/f_i(X)}$ is cofinal, as a filtered colimit of cofinal functors is cofinal~(\cite[Corollary~7.2.1.17]{kerodon}).
    \end{proof}


    \subsubsection*{Boundary objects}

    Given a geometric pattern $\cX$, we may consider a preorder on its elementary objects where $E_1\le E_2$ if there exists some inert map $E_1 \to E_2$.
    In many cases of interest, this preorder will have a maximum. We use this maximum object to define a boundary geometric pattern.

    \begin{definition}
    \label{def:maximum-elemenary}
        Let $\cX$ be a geometric pattern. We say that $\cX$ admits a \emph{maximum elementary} $M \in \cX^{\el}$ if for any $E \in \cX^{\el}$:
        \begin{itemize}
            \item there exists an inert map $E \to M$;
            \item if $M \to E$ is inert then it is an isomorphism $M \isoto E$.
        \end{itemize}
    \end{definition}

    \begin{definition}
        Let $\cX$ be a geometric pattern with a maximum elementary $M$. We say that $X \in \cX$ is \emph{boundary} if there is no inert map $M \to X$. We let $\cX_{\partial} \subseteq \cX$ be the full subcategory spanned by boundary objects. 
    \end{definition}

    \begin{remark}
        A maximum elementary $M$ is unique up to (non-canonical) isomorphism, but the space of maxima need not be contractible. Nevertheless, the full subcategory $\cX_{\partial}\subseteq \cX$ of boundary objects does not depend on the choice of $M$.
    \end{remark}

    \begin{lemma}
    \label{lem:boundary-inclusion-is-strong-segal}
        If $\cX$ has a maximum elementary, then the geometric pattern of $\cX$ restricts to a geometric pattern on $\cX_{\partial}$, and the inclusion map $i \colon \cX_{\partial} \into \cX$ is a strong Segal morphism.
    \end{lemma}
    
    \begin{proof}
        We must show that the active-inert factorization system restricts to $\cX_{\partial}$. Let $f \colon X \to Y$ be a map of boundary objects. As a map in $\cX$, it splits uniquely as a composition of an active $X \to Z$ and an inert $Z \to Y$. If $Z$ were not boundary, then there would exist an inert map $M \to Z$, and the composite $M \to Z \to Y$ would be inert, contradicting $Y \in \cX_\partial$.
    
        To see that the inclusion is strong Segal, note that by the same argument, for any boundary object $X \in \cX_{\partial}$
        \begin{equation*}
            \cX^{\el}_{/X} \;\simeq\; (\cX_{\partial})^{\el}_{/X}.
        \end{equation*}
    \end{proof}
    
    \begin{definition}
    \label{def:boundary-and-inclusion}
        Let $\cX$ be a geometric pattern with a maximum elementary. For any $\cC\in \PrL$, we denote the adjunction corresponding to $i\colon \cX_\partial\into \cX$ by
        \begin{equation*}
            \iota \coloneqq i_! \;\colon\; \Shv(\cX_{\partial};\cC) \;\longadj\; \Shv(\cX;\cC) \;\cocolon\; i^* \eqqcolon \partial.
        \end{equation*}
        We call $\partial$ the \emph{boundary} functor.
    \end{definition}

    \begin{remark}
        When the geometric pattern in question models a kind of $n$-category, the boundary functor $\partial$ is also commonly called the \emph{core} functor, since it returns the underlying $(n-1)$-categorical part.
    \end{remark}

    \begin{lemma}
    \label{lem:boundary-of-product}
        Let $\cX$ and $\cY$ be geometric patterns with maximum elementaries, then $\cX\times \cY$ also has a maximum elementary.
        Moreover, the following is a pushout square in $\GPatSeg$:
        \begin{equation*}
            \begin{tikzcd}
                {\cX_\partial\times\cY_\partial} & {\cX\times\cY_\partial} \\
                {\cX_\partial\times\cY} & {(\cX\times\cY)_\partial}
                \arrow[hook, from=1-1, to=1-2]
                \arrow[hook, from=1-1, to=2-1]
                \arrow[hook, from=1-2, to=2-2]
                \arrow[hook, from=2-1, to=2-2]
                \arrow["\lrcorner"{anchor=center, pos=0.125, rotate=180}, draw=none, from=2-2, to=1-1]
            \end{tikzcd}
        \end{equation*}
        and this pushout is preserved by $\Shv\colon \GPatSeg\to \PrL$.
    \end{lemma}

    \begin{proof}
        Let $M\in \cX^\el$ and $N\in \cY^\el$ be maximum elementary objects.
        Then $(M,N)\in (\cX\times\cY)^\el$ is a maximum elementary, since inert maps in $\cX\times \cY$ are coordinate-wise inert.
        To prove the pushout, we will show that the conditions of~\cref{prop:gpat-colimit-segal-sheaves} hold.

        Note that $(X,Y)\in \cX\times \cY$ is boundary if and only if either $X$ or $Y$ is boundary; this gives us the above pushout square in $\Cat$. Moreover, as elementary, active and inert are restricted from the product, it is also a pushout square in $\GPat$.
        It remains to prove that $\cX_\partial\times \cY \to (\cX\times \cY)_\partial$ and $\cX\times \cY_\partial \to (\cX\times \cY)_\partial$ are strong Segal.
        We consider the first; the second is symmetric, swapping the roles of $\cX$ and $\cY$.
        For every $(X,Y)\in \cX_\partial\times \cY$, we have isomorphisms by the proof of~\cref{lem:boundary-inclusion-is-strong-segal}
        \begin{equation*}
            ((\cX\times \cY)_\partial)^\el_{/(X,Y)} 
            \;\simeq\; (\cX\times \cY)^\el_{/(X,Y)}
            \;\simeq\; \cX^\el_{/X}\times \cY^\el_{/Y}
            \;\simeq\; (\cX_\partial)^\el_{/X}\times \cY^\el_{/Y}
            \;\simeq\; (\cX_\partial\times \cY)^\el_{/(X,Y)}.
        \end{equation*}
    \end{proof}

    \begin{corollary}
    \label{cor:boundary-of-external-product}
        Let $\cX$ and $\cY$ be geometric patterns with maximum elementaries, and let $X\in \Shv(\cX)$ and $Y\in \Shv(\cY)$. 
        Then the following is a pushout square in $\Shv(\cX\times\cY)$:
        \begin{equation*}
            \begin{tikzcd}
                {\iota\partial X\boxtimes \iota\partial Y} & {X\boxtimes \iota\partial Y} \\
                {\iota\partial X\boxtimes Y} & {\iota\partial (X\boxtimes Y).}
                \arrow[from=1-1, to=1-2]
                \arrow[from=1-1, to=2-1]
                \arrow[from=1-2, to=2-2]
                \arrow[from=2-1, to=2-2]
                \arrow["\lrcorner"{anchor=center, pos=0.125, rotate=180}, draw=none, from=2-2, to=1-1]
            \end{tikzcd}
        \end{equation*}
    \end{corollary}

    \begin{proof}
        Taking sheaves with the $(-)^*$-functoriality,~\cref{lem:boundary-of-product} implies the following pullback:
        \begin{equation*}
            \Shv((\cX\times\cY)_\partial) \;\isotoo\; \Shv(\cX_\partial\times \cY)\times_{\Shv(\cX_\partial\times \cY_\partial)}\Shv(\cX\times \cY_\partial) \qin \CAT.
        \end{equation*}
        Under this identification, the functor $\partial \colon \Shv(\cX\times\cY)\to \Shv((\cX\times\cY)_\partial)$ sends $X\boxtimes Y$ to the element of the pullback given by the triple 
        \begin{equation*}
            (\partial X\boxtimes Y,\; \partial X\boxtimes \partial Y,\; X\boxtimes\partial Y),
        \end{equation*}
        while the left adjoint $\iota\colon  \Shv((\cX\times\cY)_\partial)\to \Shv(\cX\times\cY)$ sends a triple $(X', Z', Y')$ to 
        \begin{equation*}
            (\iota \otimes \id) X' \sqcup_{(\iota \otimes \iota) Z'} (\id \otimes \iota) Y',
        \end{equation*}
        using the identification of~\cref{cor:product-of-geometric-patterns}.
    \end{proof}
    
     In the next proposition, we give sufficient conditions for $\iota\colon \Shv(\cX_\partial;\cC)\to \Shv(\cX;\cC)$ to be fully faithful. These conditions are in no way \emph{necessary}; see~\cref{rem:inclusion-of-grids} for example.

    \begin{proposition}
    \label{prop:core-is-localization}
        Let $\cX$ be a geometric pattern with a maximum elementary such that for any active morphism $X\to Y\in\cX$, if $X$ is boundary then so is $Y$.
        Then for any $\cC\in \PrL$, 
        \begin{equation*}
            \iota \;\colon\; \Shv(\cX_\partial;\cC) \;\too\; \Shv(\cX;\cC)
        \end{equation*}
        is fully faithful.
    \end{proposition}

    \begin{proof}
        It suffices to prove that the right adjoint $\partial$ has a further right adjoint which is fully faithful (see e.g.~\cite[Lemma~1.2]{Haine-Ramzi-Steinebrunner-2025-fully-faithful}).
        The assumption on $\cX$ implies that the inclusion $i\colon X_\partial \into X$ satisfies \emph{unique lifting of active morphisms}, as per~\cite[Definition~6.1]{Chu-Haugseng-2021-algebraic-patterns}.
        Thus, by~\cite[Proposition~6.3]{Chu-Haugseng-2021-algebraic-patterns}, right Kan extension along $i$ restricts to a functor on Segal sheaves
        \begin{equation*}
            i_* \;\colon\; \Shv(\cX_\partial;\cC) \;\too\; \Shv(\cX;\cC)
        \end{equation*}
        which is right adjoint to $i^*=\partial$.
        Since $i$ is fully faithful, the proof of \cite[Proposition~4.3.2.17]{Lurie-HTT} implies that $i_*$ is also fully faithful.
    \end{proof}

\subsection{$n$-Uple Segal spaces}

    In \cite{Rezk-2001-Segal-spaces}, Rezk introduced \emph{Segal spaces} as $\spc$-valued sheaves on the simplex category $\deltaCat$.
    
    \begin{definition}[{\cite[Example~3.3]{Chu-Haugseng-2021-algebraic-patterns}}]
        Equip $\deltaCat$ with the following geometric pattern:
        \begin{itemize}
            \item An order-preserving map $f \colon [a] \to [b]$ is active if it preserves the endpoints.
            \item An order-preserving map $f \colon [a] \to [b]$ is inert if it is the inclusion of an interval, i.e.\ $f(i) = f(0) + i$.
            \item The elementary objects are $[0]$ and $[1]$.
        \end{itemize}
    \end{definition}
    
    \begin{example}
        $\sCat \coloneqq \Shv(\deltaCat; \Set)$ is the $(1,1)$-category of strict categories.
    \end{example}

    By~\cref{cor:product-of-geometric-patterns}, $\deltaCat^{n} = \deltaCat \times \cdots \times \deltaCat$ is a geometric pattern. 
    Segal sheaves on $\deltaCat^n$ are called \emph{$n$-uple Segal spaces}.
    We now spell out the structure of this geometric pattern.
    
    \begin{notation}
        We denote elements in $\ZZ_{\ge 0}^{n}$ as vectors $\vec{a} = (a_1, \dots, a_n)$. We similarly denote the objects of $\deltaCat^n$ by 
        \begin{equation*}
            [\vec{a}] \;=\; [a_1,\dots,a_n] \;\coloneqq\; ([a_1],\dots,[a_{n}]).
        \end{equation*}

        For $X \in \PSh(\deltaCat^{n})$, we usually denote its evaluation at $[\vec{a}]$ by $X_{\vec{a}} \coloneq X([\vec{a}])$.
    \end{notation}

    \begin{definition}
    \label{def:simplicial-cubes}
        For $I \subseteq \{1,\dots, n\}$, define $[1]^I \coloneqq [\mbbm{1}^I] \in \deltaCat^{n}$ by
        \begin{equation*}
            \mbbm{1}^I_i \;=\; \begin{cases}
                1, & i \in I, \\
                0, & i \notin I.
            \end{cases}
        \end{equation*}
        We also denote $[1]^k \coloneqq [1]^{\{1,\dots,k\}}$.
    \end{definition}

    \begin{remark}
        Using the external product $\Shv(\deltaCat)^{\otimes n}\isoto \Shv(\deltaCat^n)$, we can express 
        \begin{equation*}
            [1]^n\simeq [1]\boxtimes\dots\boxtimes[1].
        \end{equation*}
        More generally, $[1]^I$ is given by the appropriate external product of $[1]$'s and $[0]$'s.
    \end{remark}

    The geometric pattern on $\deltaCat^{n}$ is described as follows:
    \begin{itemize}
        \item A morphism $f \colon [\vec{a}] \to [\vec{b}]$ is active if it is active in each coordinate, i.e.\ 
        \begin{equation*}
            \big(f(i_1,\dots, i_{j-1}, 0, i_{j+1}, \dots, i_{n})\big)_j \;=\; 0, 
            \qquad 
            \big(f(i_1,\dots, i_{j-1}, a_j, i_{j+1}, \dots, i_{n})\big)_j \;=\; b_j. 
        \end{equation*}
        \item A morphism $f \colon [\vec{a}] \to [\vec{b}]$ is inert if it is the inclusion of a grid, i.e.\ $f(\vec{i}) = f(\vec{0}) + \vec{i}$.
        \item The elementary objects are $[1]^I$ for $I \subseteq \{1,\dots, n\}$.
    \end{itemize}

    By~\cite[Example~14.21]{Chu-Haugseng-2021-algebraic-patterns}, the geometric pattern $\deltaCat^n$ is saturated.
        
    \subsubsection*{Boundary}
    The geometric pattern $\deltaCat^n$ has a maximum elementary $[1]^n$. 
    The boundary objects $\deltaCat^n_\partial\subseteq \deltaCat^n$ are those $[\vec{a}]\in \deltaCat^n$ such that there exists $1\le j\le n$ with $a_j = 0$.
    In particular, we get an adjunction as in~\cref{def:boundary-and-inclusion}:
    \begin{equation*}
        \iota \;\colon\; \Shv(\deltaCat^n_{\partial};\cC) \;\longadj\; \Shv(\deltaCat^n;\cC) \;\cocolon\; \partial.
    \end{equation*}

    \begin{corollary}
    \label{cor:uple-inclusion-fully-faithful}
        For any $\cC\in \PrL$, the functor     
        \begin{equation*}
            \iota \;\colon\; \Shv(\deltaCat^n_\partial;\cC) \;\too\; \Shv(\deltaCat^{n};\cC)
        \end{equation*}
        is fully faithful.
    \end{corollary}

    \begin{proof}
        It suffices to check that $\deltaCat^n$ satisfies the assumptions of~\cref{prop:core-is-localization}.
        Indeed, if $f \colon [\vec{a}] \to [\vec{b}]$ is active and $a_j = 0$, then
        \begin{equation*}
            0 \;=\; \big(f(i_1,\dots,i_{j-1},0,i_{j+1}, \dots,i_{n})\big)_j \;=\; b_j.
        \end{equation*}
    \end{proof}

    \begin{example}
        When $n=1$, the boundary $\partial \colon \Shv(\deltaCat)\to \spc$ is given by evaluation at $[0]$.
        The left adjoint $\iota\colon \spc\to \Shv(\deltaCat)$ is the constant Segal space functor.
    \end{example}

    \subsubsection*{Univalence}
    A Segal space $X$ behaves much like a category, with space of objects $X_0$ and space of arrows $X_1$.
    The degeneracy $X_0\to X_1$ provides identity arrows, while the Segal condition determines composition
    \begin{equation*}
        X_1\times_{X_0} X_1 \;\xlongleftarrow{\sim}\; X_2\;\xtoo{d_1} \; X_1.
    \end{equation*}
    In particular, one obtains a notion of isomorphism in $X$ as arrows with left and right inverses.
    On the other hand, there is also a notion of isomorphism given by paths in the space $X_0$.
    For a Segal space to model a category, one requires that these two notions coincide. 

    \begin{definition}
        Let $X \in \Shv(\deltaCat)$. We say that $X$ is \emph{univalent} (a.k.a.\ complete) if the following is a pullback 
        \begin{equation*}
            \begin{tikzcd}
                {X_0} & {X_3} \\
                {X_0 \times X_0} & {X_2 \times X_2.}
                \arrow[from=1-1, to=1-2]
                \arrow[from=1-1, to=2-1]
                \arrow["\lrcorner"{anchor=center, pos=0.125}, draw=none, from=1-1, to=2-2]
                \arrow["{d_3 \times d_0}", from=1-2, to=2-2]
                \arrow[from=2-1, to=2-2]
            \end{tikzcd}
        \end{equation*}
        We denote the full subcategory of univalent sheaves by $\Shv(\deltaCat)\univ$.
        Equivalently, it is the localization of $\Shv(\deltaCat)$ at the map $J\to \pt$, where $J\in \Shv(\deltaCat)$ is the walking invertible arrow.
    \end{definition}

    Univalent Segal spaces are a model for categories, $\Shv(\deltaCat)\univ\simeq \Cat$.

    \begin{definition}
        Let $X\in \Shv(\deltaCat^n)$. We say that $X$ is \emph{univalent} if 
        \begin{equation*}
            X_{a_1,\dots,a_{j-1},\bullet,a_{j+1},\dots,a_{n}}\qin \Shv(\deltaCat)
        \end{equation*}
        is a univalent Segal space for any choice of
        $a_i \in \ZZ_{\ge 0}$ and $1 \le j \le n$.
        We denote the full subcategory of univalent $n$-uple Segal spaces by $\Shv(\deltaCat^n)\univ$
    \end{definition}

    Univalent $n$-uple Segal spaces are a model for $n$-uple categories, 
    \begin{equation*}
        \Shv(\deltaCat^n)\univ \;\simeq\; (\Shv(\deltaCat)\univ)^{\otimes n} \;\simeq\; \Cat^{\otimes n}.
    \end{equation*}
    
\subsection{$\globe$-Spaces}
    Denote by $\globe  \in \Cat_{(1,1)}$
    Joyal's category of globular sums. 
    It was first defined in an unpublished work of Joyal~\cite{Joyal-1997-Theta}, and later, using an iterated categorical wreath product with $\deltaCat$, by Berger~\cite{Berger-2007-Theta-wreath}. 
    
    \begin{definition}
        Let $\cC \in \Cat_{(1,1)}$. The wreath product $\deltaCat \wr \cC \in \Cat_{(1,1)}$ is defined as follows:
        \begin{itemize}
            \item Objects are tuples $([a]; x_1, \dots, x_a)$ for $[a] \in \deltaCat$ and $x_i \in \cC$.
            \item Morphisms $([a]; x_1, \dots, x_a) \to ([b]; y_1, \dots, y_b)$ consist of a map $\alpha \colon [a] \to [b]$ in $\deltaCat$ together with morphisms $f_{i,j} \colon x_i \to y_j$ in $\cC$ for every pair $(i,j)$ such that $\alpha(i-1) < j \le \alpha(i)$.
        \end{itemize}
    \end{definition}
    
    We now inductively define $\globe[-1] \coloneqq \emptyset$ to be the empty category and $\globe \coloneqq \deltaCat \wr \globe[n-1]$. In particular, $\globe[0] = \pt$ is the terminal category and $\globe[1] = \deltaCat$.
    The inclusion $\emptyset \into \pt$ induces a fully faithful embedding $\globe[n-1]\into \globe[n]$, which we view as a subcategory.

    We define the \emph{$n$-globe} $\laxdisk{n}\in \globe$ (also known as the walking $n$-morphism, $n$-cell, or $n$-disk) inductively by
    \begin{equation*}
        \laxdisk{0}\coloneqq [0]\in \globe[0],\qquad \laxdisk{n} \coloneqq ([1]; \laxdisk{n-1})\in\globe.
    \end{equation*}
    Considering  $\laxdisk{k},\laxdisk{k+1}\in \globe$ for $k<n$, there are two \emph{boundary} maps $\delta^-, \delta^+\colon \laxdisk{k}\to \laxdisk{k+1}$, induced from the two maps $[0] \to [1]$, and one \emph{degeneracy} map $\sigma \colon \laxdisk{k+1} \to \laxdisk{k}$, induced from $[1] \to [0]$. The full subcategory of globes in $\globe$ can be described diagrammatically as
    \begin{equation*}
        \begin{tikzcd}
            {\laxdisk{0}} & {\laxdisk{1}} & \cdots & {\laxdisk{n}.}
            \arrow[shift right=2, from=1-1, to=1-2]
            \arrow[shift left=2, from=1-1, to=1-2]
            \arrow[from=1-2, to=1-1]
            \arrow[shift left=2, from=1-2, to=1-3]
            \arrow[shift right=2, from=1-2, to=1-3]
            \arrow[from=1-3, to=1-2]
            \arrow[shift left=2, from=1-3, to=1-4]
            \arrow[shift right=2, from=1-3, to=1-4]
            \arrow[from=1-4, to=1-3]
        \end{tikzcd}
    \end{equation*}

    Chu and Haugseng defined an algebraic pattern on $\globe \op$, with a factorization system originally due to Berger~\cite{Berger-2002-Theta}, which we consider as a geometric pattern.
    \begin{definition}[{\cite[Example~3.5]{Chu-Haugseng-2021-algebraic-patterns}}]
        Equip $\globe $ with the following inductively defined geometric pattern:
        \begin{itemize}
            \item A map $(\alpha; (f_{i,j})) \colon ([a]; x_1, \dots, x_a) \to ([b]; y_1, \dots, y_{b})$ is active (resp.\ inert) if $\alpha$ and $f_{i,j}$ are all active (resp.\ inert).
            \item The elementary objects are the globes $\laxdisk{0}, \dots, \laxdisk{n}$.
        \end{itemize}
    \end{definition}

    By~\cite[Example~14.21]{Chu-Haugseng-2021-algebraic-patterns}, the geometric pattern $\globe$ is saturated.

    \begin{example}
        $\sCatn\coloneqq \Shv(\globe;\Set)$ is the category of \emph{strict $n$-categories} (\cite{Berger-2002-Theta}).
    \end{example}

    \begin{example}
    \label{exm:Theta_n-spaces}
        We define $\globe$-spaces to be $\spc$-valued sheaves on $\globe$. By~\cite{Ayala-Francis-2018-flagged-categories}, they can be identified with \emph{flagged $n$-categories}.
    \end{example}

    \begin{lemma}
    \label{lem:globe-is-a-weak-generator}
        The $n$-globe $\laxdisk{n}$ is a weak generator of $\Shv(\globe)$; that is, it generates $\Shv(\globe)$ under colimits.
    \end{lemma}
    
    \begin{proof}
        By the Segal condition, $\Shv(\globe)$ is generated under colimits by the elementary objects $\laxdisk{0}, \dots, \laxdisk{n}$. Since every $\laxdisk{k}$ is a retract of $\laxdisk{n}$ for $k \le n$, the single object $\laxdisk{n}$ suffices to generate the category.
    \end{proof}

    \subsubsection*{Inclusion and core}

    The geometric pattern $\globe$ has a maximum elementary $\laxdisk{n}$, and the boundary inclusion
    $(\globe)_\partial \into \globe$ agrees with the inclusion $\globe[n-1]\into \globe$.
    In particular, this inclusion is a strong Segal morphism by~\cref{lem:boundary-inclusion-is-strong-segal} (see also~\cite[Example~4.11]{Chu-Haugseng-2021-algebraic-patterns}).
    Consider the adjunction as in~\cref{def:boundary-and-inclusion}
    \begin{equation*}
        \iota \;\colon\; \Shv(\globe[n-1];\cC) \;\longadj\; \Shv(\globe;\cC) \;\cocolon\; \partial.
    \end{equation*}
    
    The following lemma is well known. We include the proof for completeness.
    \begin{lemma}
    \label{lem:inclusion-of-fold}
        For any $\cC\in \PrL$, the functor
        \begin{equation*}
            \iota \;\colon\; \Shv(\globe[n-1]; \cC) \;\too\; \Shv(\globe; \cC)
        \end{equation*}
        is fully faithful. 
    \end{lemma}

    \begin{proof}
        It suffices to check that $\globe$ satisfies the assumptions of~\cref{prop:core-is-localization}.
        We prove inductively that whenever $f\colon x\to y\in \globe$ is an active map and $x \in \globe[n-1]$, then also $y\in \globe[n-1]$.
        For $n=0$, it is trivial.
        Write $x=([a];x_1,\dots,x_a)$, $y=([b];y_1,\dots,y_b)$, and $f=(\alpha,(f_{ij}))$. 
        Since $\alpha\colon [a]\to [b]$ is active, for every $0\le j\le b$, we can find some $1 \le i\le a$ such that $\alpha(i-1)\le j\le \alpha(i)$.
        Thus, we have an active map $f_{ij}\colon x_i\to y_j$, and since $x_i\in \globe[n-2]$, it follows by induction that also $y_j\in \globe[n-2]$.
    \end{proof}

    \begin{example}
        For $\cC = \Set$, the functor $\partial \colon \sCatn \to \sCatn[(n-1)]$ returns the $(n-1)$-core, i.e., the maximal strict $(n-1)$-subcategory. Its left adjoint is the inclusion of strict $(n-1)$-categories into strict $n$-categories.
    \end{example}

    Since the inclusions
    \begin{equation*}
        \globe[-1] \;\intoo\; \globe[0] \;\intoo\; \globe[1] \;\intoo\; \globe[2] \;\intoo\; \dots
    \end{equation*}
    are all strong Segal morphisms, by~\cref{cor:gpat-filtered-colimit-strong-segal} they have a colimit in $\GPatSeg$.

    \begin{definition}
        Define the geometric pattern
        \begin{equation*}
            \globe[] \;=\; \globe[\infty] \;\coloneqq\; \colim_n \globe[n] \qin \GPatSeg.
        \end{equation*}
    \end{definition}

    In particular,~\cref{cor:gpat-filtered-colimit-strong-segal} implies that
    \begin{equation*}
        \colim_n \Shv(\globe;\cC) \;\isotoo\; \Shv(\globe[];\cC) \qin \PrL,
    \end{equation*}
    or equivalently, by passing to right adjoints, 
    \begin{equation*}
        \Shv(\globe[];\cC) \;\isotoo\; \lim_n \Shv(\globe;\cC) \qin \CAT.
    \end{equation*}
    \begin{example}
        $\sCatoo \coloneqq \Seg(\globe[];\Set)$ is the category of strict $\infty$-categories.
    \end{example}
    
    \begin{example}
        Following \cref{exm:Theta_n-spaces}, we define a $\globe[]$-space to be a sheaf on $\globe[]$. The category of $\globe[]$-spaces is identified with the category of flagged $\infty$-categories.
    \end{example}

    \subsubsection*{The suspension functor}

    Consider the functor $\globe[n-1]\to \globe$ given by $x\mapsto ([1];x)$. 
    Post-composition with the Yoneda embedding
    \begin{equation*}
        \globe[n-1]\;\too\; \globe\;\intoo\; \PSh(\globe)
    \end{equation*}
    can be canonically lifted to bi-pointed presheaves
    \begin{equation*}
        \globe[n-1]\;\too\; \PSh(\globe)_{*,*},
    \end{equation*}
    where $([1];x)$ is pointed by the two maps from $[0]$ selecting the two faces $[0]\to [1]$.

    \begin{definition}
        The \emph{bi-pointed suspension} functor
        \begin{equation*}
            \suspensionbi\colon \PSh(\globe[n-1])\;\too\; \PSh(\globe)_{*,*}
        \end{equation*}
        is the left Kan extension of $\globe[n-1]\to\PSh(\globe)_{*,*}$.
        The \emph{suspension} functor is then given by further forgetting the bi-pointing 
        \begin{equation*}
            \suspension\colon \PSh(\globe[n-1])\;\xtoo{\suspensionbi}\; \PSh(\globe)_{*,*}\;\too\; \PSh(\globe).
        \end{equation*}
    \end{definition}

    \begin{proposition}[{\cite[Theorem~2.25]{Campion-2023-pasting}}]
        Suspension preserves Segal sheaves, as well as \mbox{$\Set$-valued} Segal sheaves, descending to functors
        \begin{equation*}
            \suspension \colon \Shv(\globe[n-1])\;\too\; \Shv(\globe),
            \qquad
            \suspension \colon \Shv(\globe[n-1];\Set)\;\too\; \Shv(\globe;\Set).
        \end{equation*}
    \end{proposition}

    \begin{example}
    \label{exm:suspension-of-globe}
        On a representable $x\in \globe[n-1]$, the suspension is given by $\suspension x\simeq([1];x)$. In particular, $\suspension \laxdisk{k}\simeq\laxdisk{k+1}$.
    \end{example}

    \begin{example}
    \label{exm:suspension-of-a-boundary-of-a-globe}
        The suspension of the boundary of the $k$-globe is the boundary of the $(k+1)$-globe, $\suspension \iota \partial \laxdisk{k}\simeq \iota\partial\laxdisk{k+1}$.
    \end{example}

    \begin{definition}
        The \emph{hom functor}
        \begin{equation*}
            \hom\colon \PSh(\globe)_{*,*}\;\too\; \PSh(\globe[n-1])
        \end{equation*}
        is the right adjoint of the bi-pointed suspension.
        For $(X,s,t) \in \PSh(\globe)_{*,*}$, we denote $\hom_X(s,t)\coloneqq \hom(X,s,t)$.
    \end{definition}

     Explicitly, for $(X,s,t)\in \PSh(\globe)_{*,*}$, the hom is given by the pullback
    \begin{equation*}
        \begin{tikzcd}
            { \hom_X(s,t)} & \pt \\
            {X_{([1];\bullet)}} & {X_{[0]}^{\times 2}.}
            \arrow[from=1-1, to=1-2]
            \arrow[from=1-1, to=2-1]
            \arrow["\lrcorner"{anchor=center, pos=0.125}, draw=none, from=1-1, to=2-2]
            \arrow["{{(s,t)}}", from=1-2, to=2-2]
            \arrow[from=2-1, to=2-2]
        \end{tikzcd}
    \end{equation*}

    \begin{corollary}
    \label{cor:map-from-suspension}
        Let $Y\in \PSh(\globe[n-1])$ and $X\in\PSh(\globe)$. 
        Then the following is a pullback square:
        \begin{equation*}
            \begin{tikzcd}
                {\Map_{\PSh(\globe)}(\suspension Y,X)} & {X_{[0]}^{\times 2}} \\
                {\Map_{\PSh(\globe[n-1])}(Y,X_{([1];\bullet)})} & {\Map_{\PSh(\globe[n-1])}(Y,X_{[0]}^{\times 2}).}
                \arrow[from=1-1, to=1-2]
                \arrow[from=1-1, to=2-1]
                \arrow[from=1-2, to=2-2]
                \arrow[""{name=0, anchor=center, inner sep=0}, from=2-1, to=2-2]
                \arrow["\lrcorner"{anchor=center, pos=0.125}, draw=none, from=1-1, to=0]
            \end{tikzcd}
        \end{equation*}
    \end{corollary}

    \begin{proof}
        A map $\suspension Y\to X$ is given by a choice of two points $(s,t)\in X_{[0]}^{\times 2}$ and a bi-pointed map $\suspensionbi Y\to (X,s,t)$.
        More precisely, we have isomorphisms
        \begin{align*}
            \Map_{\PSh(\globe)}(\suspension Y,X)
            &\simeq \colim_{(s,t)\in X_{[0]}^{\times 2}} \Map_{\PSh(\globe)_{*,*}}(\suspensionbi Y,(X,s,t))\\
            &\simeq \colim_{(s,t)\in X_{[0]}^{\times 2}} \Map_{\PSh(\globe[n-1])}(Y,\hom_X(s,t)).
        \end{align*}
        The result then follows from the above pullback description of $\hom_X(s,t)$.
    \end{proof}

    \subsection{$n$-Fold Segal spaces}

    In \cite{Barwick-2005-n-categories}, Barwick introduced a model for $n$-categories using \emph{$n$-fold Segal spaces}, defined as a full subcategory of $n$-uple Segal spaces. 
    This subcategory can be identified with the essential image of an embedding of $\globe$-spaces.

    \begin{definition}
        Given $\cC \in \Cat_{(1,1)}$, there exists a functor $\G_{\cC} \colon \deltaCat \times \cC \to \deltaCat \wr \cC$ defined by
        \begin{equation*}
            \G_{\cC}([a], x) \;\coloneqq\; ([a]; x, x, \dots, x).
        \end{equation*}
    \end{definition}

    By iteration, this induces a functor 
    \begin{equation*}
        \G \colon \deltaCat^{n} \;\too\; \globe.
    \end{equation*}
    Explicitly, $[\vec{a}]\in \deltaCat^n$ is sent to the globular sum with $(a_1+1)$ objects, between each two consecutive objects $(a_2+1)$ $1$-morphisms, between each two consecutive $1$-morphisms $(a_3+1)$ $2$-morphisms, and so on.

    \begin{example}
        The globes are described by $\G[1]^k\simeq \laxdisk{k}$.
    \end{example}

    By~\cite[Examples~4.10 and~4.11]{Chu-Haugseng-2021-algebraic-patterns}, $\G \colon \deltaCat^{n} \to \globe$ is a Segal morphism of geometric patterns.
    The induced functor
    \begin{equation*}
        \G^* \;\colon\; \Shv(\globe) \;\too\; \Shv(\deltaCat^{n})
    \end{equation*}
    is fully faithful by~\cite[Corollary~4.2]{Haugseng-2018-equivalence}, with essential image the \emph{$n$-fold Segal spaces}.
    \begin{definition}
        Let $X\in \Seg(\deltaCat^n)$ be an $n$-uple Segal space. We call $X$ an \emph{$n$-fold Segal space} if it satisfies the following conditions:
        \begin{enumerate}
            \item $X_{a, \bullet, \dots, \bullet} \in \Shv(\deltaCat^{n-1})$ is an $(n-1)$-fold Segal space for all $a$,
            \item $X_{0,\dots,0}\to X_{0, \bullet, \dots, \bullet}$ is an isomorphism.
        \end{enumerate}
    \end{definition}

    By~\cite[Proposition~4.12]{Haugseng-2018-iterated}, $\G^*$ admits a right adjoint
    \begin{equation*}
        \G_* \;\colon\; \Shv(\deltaCat^{n}) \;\too\; \Shv(\globe)
    \end{equation*}
    which morally discards all cells that are non-degenerate in the constancy directions.
    
    We can extract an inductive description of $\G_*X$ for $X\in \Shv(\deltaCat^n)$ from the proof of loc.\ cit.\ (see also~\cite[Remark~4.13]{Haugseng-2018-iterated}):
    \begin{lemma}
    \label{lem:inductive-G_*}
        Let $X \in \Shv(\deltaCat^{n})$. Then
        \begin{equation*}
            \begin{tikzcd}
                {(\G_*X)_{([a];\bullet)}} & {X_{0,\dots,0}^{\times (a+1)}} \\
                {\G_*(X_{a,\bullet,\dots,\bullet})} & {\G_*(X_{0,\bullet,\dots,\bullet})^{\times (a+1)}}
                \arrow[from=1-1, to=1-2]
                \arrow[from=1-1, to=2-1]
                \arrow[from=1-2, to=2-2]
                \arrow[""{name=0, anchor=center, inner sep=0}, from=2-1, to=2-2]
                \arrow["\lrcorner"{anchor=center, pos=0.125}, draw=none, from=1-1, to=0]
            \end{tikzcd}
            \qin \Shv(\globe[n-1])
        \end{equation*}
        where $(\G_*X)_{([a];\bullet)}$ is the preasheaf $x\mapsto (\G_*X)_{([a];x,\dots,x)}$.
    \end{lemma}

    For $Y\in \Shv(\globe[n-1])$, the counit of the bi-pointed suspension induces a map
    \begin{equation*}
        Y\;\isotoo\; \hom(\suspensionbi Y) \;\too\; (\suspension Y)_{([1];\bullet)}\qin \Shv(\globe[n-1]).
    \end{equation*}
    After applying $\G^*$, we get a map $\G^*Y\to (\G^*\suspension Y)_{1,\bullet,\dots,\bullet}$, which corresponds to
    \begin{equation*}
        [1]\boxtimes\G^* Y\;\too\; \G^*\suspension Y\qin \Shv(\deltaCat^n).
    \end{equation*}
    
    \begin{lemma}
    \label{lem:suspension-n-fold}
        Let $Y\in \Shv(\globe[n-1])$, then
        the following is a pushout square in $\Shv(\deltaCat^n)$:
        \begin{equation*}
            \begin{tikzcd}
                {\iota\partial[1]\boxtimes\G^*Y} & {\iota\partial[1]\boxtimes\pt} \\
                {[1]\boxtimes\G^*Y} & {\G^*\suspension Y.}
                \arrow[""{name=0, anchor=center, inner sep=0}, from=1-1, to=1-2]
                \arrow[from=1-1, to=2-1]
                \arrow[from=1-2, to=2-2]
                \arrow[from=2-1, to=2-2]
                \arrow["\lrcorner"{anchor=center, pos=0.125, rotate=180}, draw=none, from=2-2, to=0]
            \end{tikzcd}
        \end{equation*}
    \end{lemma}

    \begin{proof}
        Let $X\in \Shv(\deltaCat^n)$ and consider
        \begin{equation*}
            \Map_{\Shv(\deltaCat^n)}(\G^*\suspension Y,X)
            \;\simeq\; \Map_{\Shv(\globe)}(\suspension Y,\G_*X).
        \end{equation*}
        By~\cref{cor:map-from-suspension}, this mapping space is given by the pullback
        \begin{equation*}
            \begin{tikzcd}
                {\Map_{\PSh(\globe)}(\suspension Y,\G_*X)} & {X_{0,\dots,0}^{\times 2}} \\
                {\Map_{\PSh(\globe[n-1])}(Y,(\G_* X)_{([1];\bullet)})} & {\Map_{\PSh(\globe[n-1])}(Y,X_{0,\dots,0}^{\times 2}).}
                \arrow[from=1-1, to=1-2]
                \arrow[from=1-1, to=2-1]
                \arrow[from=1-2, to=2-2]
                \arrow[""{name=0, anchor=center, inner sep=0}, from=2-1, to=2-2]
                \arrow["\lrcorner"{anchor=center, pos=0.125}, draw=none, from=1-1, to=0]
            \end{tikzcd}
        \end{equation*}
        Combined with the pullback describing $(\G_* X)_{[1];\bullet)}$, we get 
        \begin{equation*}
            \begin{tikzcd}
                {\Map_{\PSh(\globe)}(\suspension Y,\G_*X)} & {X_{0,\dots,0}^{\times 2}} \\
                {\Map_{\PSh(\globe[n-1])}(Y,\G_*(X_{1,\bullet,\dots,\bullet}))} & {\Map_{\PSh(\globe[n-1])}(Y,\G_*(X_{0,\bullet,\dots,\bullet})^{\times 2})}.
                \arrow[from=1-1, to=1-2]
                \arrow[from=1-1, to=2-1]
                \arrow[from=1-2, to=2-2]
                \arrow[""{name=0, anchor=center, inner sep=0}, from=2-1, to=2-2]
                \arrow["\lrcorner"{anchor=center, pos=0.125, rotate=0}, draw=none, from=1-1, to=0]
            \end{tikzcd}
        \end{equation*}
        This gives us the required pushout after adjunctions.
    \end{proof}
    
    Consider the following commuting squares of geometric patterns:
    \begin{equation*}
        \begin{tikzcd}
            {\deltaCat^{n-1}} & {\globe[n-1]} && {\deltaCat^{n}_\partial} & {\globe[n-1]} \\
            {\deltaCat^n} & \globe && {\deltaCat^n} & \globe
            \arrow["\G", from=1-1, to=1-2]
            \arrow["j"', hook, from=1-1, to=2-1]
            \arrow[hook, from=1-2, to=2-2]
            \arrow["\G", from=1-4, to=1-5]
            \arrow[hook, from=1-4, to=2-4]
            \arrow[hook, from=1-5, to=2-5]
            \arrow["\G", from=2-1, to=2-2]
            \arrow["\G", from=2-4, to=2-5]
        \end{tikzcd}
    \end{equation*}
    where $j \colon \deltaCat^{n-1}\into \deltaCat^n$ is given by $[a_1,\dots ,a_{n-1}]\mapsto [a_1,\dots ,a_{n-1},0]$.
    Taking sheaves with the $(-)^*$-functoriality, we get commuting squares
    \begin{equation*}
        \begin{tikzcd}
            {\Shv(\deltaCat^{n-1})} & {\Shv(\globe[n-1])} && {\Shv(\deltaCat^{n}_\partial)} & {\Shv(\globe[n-1])} \\
            {\Shv(\deltaCat^n)} & {\Shv(\globe)} && {\Shv(\deltaCat^n)} & {\Shv(\deltaCat^n).}
            \arrow["{{{\G^*}}}"', from=1-2, to=1-1]
            \arrow["{{\G^*}}"', from=1-5, to=1-4]
            \arrow["{j^*}", from=2-1, to=1-1]
            \arrow["\partial"', from=2-2, to=1-2]
            \arrow["{{{\G^*}}}"', from=2-2, to=2-1]
            \arrow["\partial", from=2-4, to=1-4]
            \arrow["\partial"', from=2-5, to=1-5]
            \arrow["{{\G^*}}"', from=2-5, to=2-4]
        \end{tikzcd}
    \end{equation*}

    \begin{lemma}
    \label{lem:globe-commute-boundary}
        The following squares commute by the Beck--Chevalley condition:
        \begin{equation*}
            \begin{tikzcd}[sep=scriptsize]
                {\Shv(\deltaCat^{n-1})} & {\Shv(\globe[n-1])} \\
                {\Shv(\deltaCat^n)} & {\Shv(\globe)}
                \arrow["{{{{{\G_*}}}}}", from=1-1, to=1-2]
                \arrow["{{{j^*}}}", from=2-1, to=1-1]
                \arrow["{{{{{\G_*}}}}}", from=2-1, to=2-2]
                \arrow["\partial"', from=2-2, to=1-2]
            \end{tikzcd}
            \quad
            \begin{tikzcd}[sep=scriptsize]
                {\Shv(\deltaCat^{n-1})} & {\Shv(\globe[n-1])} \\
                {\Shv(\deltaCat^n)} & {\Shv(\globe)}
                \arrow["{{{j_!}}}"', from=1-1, to=2-1]
                \arrow["{{{{\G^*}}}}"', from=1-2, to=1-1]
                \arrow["\iota", from=1-2, to=2-2]
                \arrow["{{{{\G^*}}}}"', from=2-2, to=2-1]
            \end{tikzcd}
            \quad
            \begin{tikzcd}[sep=scriptsize]
                {\Shv(\deltaCat^{n}_\partial)} & {\Shv(\globe[n-1])} \\
                {\Shv(\deltaCat^{n-1})} & {\Shv(\globe).}
                \arrow["\iota"', from=1-1, to=2-1]
                \arrow["{{{\G^*}}}"', from=1-2, to=1-1]
                \arrow["\iota", from=1-2, to=2-2]
                \arrow["{{{\G^*}}}"', from=2-2, to=2-1]
            \end{tikzcd}
        \end{equation*}
    \end{lemma}

    \begin{proof}
        For the leftmost square, let $X\in\Shv(\deltaCat^n)$.
        It is enough to prove that the Beck--Chevalley comparison map
        \begin{equation*}
            \partial \G_* X\;\too\; \G_*j^*X \qin \Shv(\globe[n-1])
        \end{equation*}
        is an isomorphism after evaluating at the weak generator $\laxdisk{n-1}$.
        This follows inductively from the description of $(\G_* X)_{([1];\bullet)}$ in~\cref{lem:inductive-G_*}.
        
        The second square follows from the first square by taking left adjoints.
        
        For the third square, let $Y\in \Shv(\globe[n-1])$. 
        Since $\iota\colon \Shv(\deltaCat^n_\partial)\into \Shv(\deltaCat^n)$ is fully faithful
        (\cref{cor:uple-inclusion-fully-faithful}), it is enough to check that $\G^*\iota Y$ is in its essential image.
        This follows from the second square, since $j_!$ factors through $\iota$.
    \end{proof}
    
    Consider the map $[1]^n \to \G^*\laxdisk{n}$ which selects the identity cell $\laxdisk{n}\to \laxdisk{n}$.
    Taking the boundary, we get a map $\iota\partial[1]^n \to \iota\partial\G^*\laxdisk{n}\simeq \G^*\iota\partial\laxdisk{n}$.

    \begin{lemma}
    \label{lem:n-uple-globe-pushout}
        The following is a pushout square in $\Seg(\deltaCat^n)$:
        \begin{equation*}\begin{tikzcd}
            {\iota\partial[1]^n } & {\G^*\iota\partial\laxdisk{n}} \\
            {[1]^n } & {\G^*\laxdisk{n}}
            \arrow[from=1-1, to=1-2]
            \arrow[from=1-1, to=2-1]
            \arrow[from=1-2, to=2-2]
            \arrow[from=2-1, to=2-2]
            \arrow["\lrcorner"{anchor=center, pos=0.125, rotate=180}, draw=none, from=2-2, to=1-1]
        \end{tikzcd}\end{equation*}
    \end{lemma}
    
    \begin{proof}
        We proceed by induction on $n$, the case $n = 0$ being trivial.
        Consider the decomposition of the square above as the composite of the following two squares:
        \begin{equation*}
            \begin{tikzcd}
                {\iota\partial[1]^n } & {\iota\partial([1]\boxtimes \G^*\laxdisk{n-1})} & {\iota\partial\G^*\laxdisk{n}} \\
                {[1]^n } & {[1]\boxtimes \G^*\laxdisk{n-1}} & {\G^*\laxdisk{n}.}
                \arrow[""{name=0, anchor=center, inner sep=0}, from=1-1, to=1-2]
                \arrow[from=1-1, to=2-1]
                \arrow[""{name=1, anchor=center, inner sep=0}, from=1-2, to=1-3]
                \arrow[from=1-2, to=2-2]
                \arrow[from=1-3, to=2-3]
                \arrow[from=2-1, to=2-2]
                \arrow[from=2-2, to=2-3]
            \end{tikzcd}
        \end{equation*}
        The right square appears as the bottom square in the following commutative diagram:
        \begin{equation*}
            \begin{tikzcd}
                {\iota\partial[1]\boxtimes\iota\partial\G^*\laxdisk{n-1}} & {\iota\partial[1]\boxtimes\G^*\laxdisk{n-1}} & {\iota\partial[1]\boxtimes\pt} \\
                {[1]\boxtimes\iota\partial\G^*\laxdisk{n-1}} & {\iota\partial([1]\boxtimes \G^*\laxdisk{n-1})} & {\iota\partial\G^*\laxdisk{n}} \\
                & {[1]\boxtimes \G^*\laxdisk{n-1}} & {\G^*\laxdisk{n}.}
                \arrow[from=1-1, to=1-2]
                \arrow[from=1-1, to=2-1]
                \arrow[from=1-2, to=1-3]
                \arrow[from=1-2, to=2-2]
                \arrow[from=1-3, to=2-3]
                \arrow[from=2-1, to=2-2]
                \arrow[from=2-2, to=2-3]
                \arrow[from=2-2, to=3-2]
                \arrow[from=2-3, to=3-3]
                \arrow[from=3-2, to=3-3]
            \end{tikzcd}
        \end{equation*}
        Both the right and the top rectangles are pushouts by \cref{lem:suspension-n-fold,exm:suspension-of-globe,exm:suspension-of-a-boundary-of-a-globe}. Moreover, the leftmost square is a pushout by \cref{cor:boundary-of-external-product}. Therefore, the bottom square is a pushout.
    
        It remains to show that the original left square is a pushout. To see this, consider the following diagram:
        \begin{equation*}
            \begin{tikzcd}[sep=scriptsize]
                {\iota\partial[1]\boxtimes\iota\partial[1]^{n-1}} && {\iota\partial[1]\boxtimes\iota\partial\G^*\laxdisk{n-1}} & \\
                & {[1]\boxtimes\iota\partial[1]^{n-1}} && {[1]\boxtimes\iota\partial\G^*\laxdisk{n-1}} \\
                {\iota\partial[1]\boxtimes[1]^{n-1}} && {\iota\partial[1]\boxtimes \G^*\laxdisk{n-1}} \\
                & {\iota\partial[1]^n } && {\iota\partial([1]\boxtimes \G^*\laxdisk{n-1})} \\
                \\
                & {[1]^n } && {[1]\boxtimes \G^*\laxdisk{n-1}}
                \arrow[from=1-1, to=1-3]
                \arrow[from=1-1, to=2-2]
                \arrow[from=1-1, to=3-1]
                \arrow[from=1-3, to=2-4]
                \arrow[from=1-3, to=3-3]
                \arrow[from=2-4, to=4-4]
                \arrow[from=3-1, to=3-3]
                \arrow[from=3-1, to=4-2]
                \arrow[from=3-3, to=4-4]
                \arrow[from=4-2, to=4-4]
                \arrow[from=4-2, to=6-2]
                \arrow[from=4-4, to=6-4]
                \arrow[from=6-2, to=6-4]
                \arrow[from=2-2, to=2-4, crossing over]
                \arrow[from=2-2, to=4-2, crossing over]
            \end{tikzcd}
        \end{equation*}
        The back face and the composite of the front faces are pushouts by the inductive hypothesis. Furthermore, the left and right faces are pushouts by \cref{cor:boundary-of-external-product}, which implies that the top-front face is also a pushout. Consequently, the bottom-front face is a pushout as required.
    \end{proof}

    \subsubsection*{Univalence}

    \begin{definition}
        Let $X \in \Shv(\globe)$. We say that $X$ is \emph{univalent} if $\G^* X\in \Shv(\deltaCat^n)$ is a univalent $n$-uple Segal space. We denote the full subcategory of univalent sheaves by $\Shv(\globe)\univ$.
    \end{definition}

    Univalent $\globe$-spaces are a model for $n$-categories, $\Shv(\globe)\univ\simeq \Catn$.

    The univalence condition can also be phrased inductively.
    \begin{lemma}[{\cite[Lemma~7.22]{Haugseng-2018-iterated}}]
    \label{lem:univalent-inductive}
        A $\globe$-space $X$ is univalent if and only if:
        \begin{enumerate}
            \item $X_{([\bullet];[0])}\in \Shv(\deltaCat)$ is univalent, and
            \item $X_{([1];\bullet)}\in \Shv(\globe[n-1])$ is univalent.
        \end{enumerate}
    \end{lemma}

    \begin{corollary}
    \label{cor:univalent-fold}
        If $X \in \Shv(\deltaCat^n)$ is univalent, then $\G_* X \in \Shv(\globe)$ is also univalent.
    \end{corollary}

    \begin{proof}
        First, $(\G_*X)_{([\bullet];[0])} \simeq X_{\bullet,0,\dots,0}$ is univalent by the univalency of $X$. The univalency of $(\G_* X)_{([1];\bullet)}$ then follows from its inductive description in~\cref{lem:inductive-G_*}.
    \end{proof}

    The boundary functor sends univalent sheaves to univalent sheaves, and therefore gives rise to a functor
    \begin{equation*}
        \partial \;\colon\; \Shv(\globe)\univ \;\too\; \Shv(\globe[n-1])\univ.
    \end{equation*}
    This functor corresponds to the $(n-1)$-core, which is the right adjoint to the inclusion of $(n-1)$-categories in $n$-categories.
    Taking the limit, we acquire a model for $\infty$-categories:
    \begin{equation*}
        \Shv(\globe[])\univ \;\coloneqq\; \lim_{n}\Shv(\globe)\univ \;\simeq\; \Catn[\infty].
    \end{equation*}

\section{Lax grids}

Given $[\vec{a}]\in \deltaCat^n$, the \emph{lax grid} $\square[\vec{a}]\in \sCatn$ was constructed in~\cite[Theorem~1.3]{Al-Agl-Brown-Steiner-2002-strict-globular-cubical} as the strict $n$-category which we roughly describe as follows:
\begin{itemize}
    \item the objects are tuples $0\le \vec{b}\le \vec{a}$,
    \item the 1-morphisms are freely generated by arrows 
    \begin{equation*}
        \vec{b} \;\too\; \vec{b}+e_i,
    \end{equation*}
    \item the 2-morphisms are freely generated by lax squares
    \begin{equation*}
        \begin{tikzcd}
            {\vec{b}} & {\vec{b}+e_i} \\
            {\vec{b}+e_j} & {\vec{b}+e_i+e_j,}
            \arrow[from=1-1, to=1-2]
            \arrow[from=1-1, to=2-1]
            \arrow[between={0.3}{0.7}, Rightarrow, from=1-2, to=2-1]
            \arrow[from=1-2, to=2-2]
            \arrow[from=2-1, to=2-2]
        \end{tikzcd}
        \qquad i < j,
    \end{equation*}
    \item and so on up to dimension $n$.
\end{itemize}
In particular, $\laxcube{n}\coloneqq\square[1]^n$ is the lax $n$-cube.
This construction extends to a functor
\begin{equation*}
    \square\colon \deltaCat^n\;\too\; \sCatn,\qquad [\vec{a}]\;\mapsto\; \square[\vec{a}].
\end{equation*}

\begin{definition}
    The \emph{$n$-grid category} $\grid\subseteq \sCatn$ is the essential image of $\square\colon \deltaCat^n\to \sCatn$. 
\end{definition}

\begin{remark}
\label{rem:maps-of-grids-order}
    Not every map $\varphi \colon \square[\vec{a}] \to \square[\vec{b}]\in \grid$ is induced from a map $[\vec{a}]\to [\vec{b}]\in\deltaCat^n$.
    For instance, this fails for the connections that are introduced in \cref{subsec:maps-of-lax-cubes}, or for the isomorphism $\square[0,1] \simeq \square[1,0]$.
    Nevertheless, any such $\varphi$ necessarily preserves the partial order on objects: if $0 \le \vec{c} \le \pvec{c}' \le \vec{a}$, then $\varphi(\vec{c}) \le \varphi(\pvec{c}')$.
\end{remark}

In this section, we study lax grids and define a geometric pattern on $\grid$.
We also show that the corresponding sheaves acquire a monoidal structure, which we call the \emph{Gray product}.

\subsection{The strict Gray product}

    The \emph{strict Gray product} is a presentably monoidal structure on $\sCatoo$, which we denote $\xlaxs$, defined in~\cite{Crans-1995-gray,Al-Agl-Brown-Steiner-2002-strict-globular-cubical,Steiner-2004-gray}.
    If $\cC$ is a strict $n$-category and $\cD$ is a strict $k$-category, then $\cC\xlaxs\cD$ is a strict $(n+k)$-category.
    Moreover, if either $\cC$ or $\cD$ is a strict 0-category (i.e.\ a set), then $\cC\xlaxs\cD\simeq \cC\times \cD$.
    A defining feature of the strict Gray product is that for any $[\vec{a}]\in \deltaCat^n$ and $[\vec{b}]\in\deltaCat^k$
    \begin{equation*}
        \square[\vec{a}]\xlaxs\square[\vec{b}]\simeq \square[\vec{a},\vec{b}].
    \end{equation*}
    In particular, $\laxcube{n}\xlaxs\laxcube{k}\simeq \laxcube{n+k}$.
    \begin{corollary}
        The functor $\square\colon \deltaCat^n\to \sCatn$ is given by the strict Gray product 
        \begin{equation*}
            \square[\vec{a}]\simeq [a_1]\xlaxs [a_2]\xlaxs\dots\xlaxs [a_n].
        \end{equation*}
    \end{corollary}

    The suspension of strict categories can be described in terms of the strict Gray product, in an analogous way to \cref{lem:suspension-n-fold}.

    \begin{lemma}[{\cite[Corollary~B.6.6]{Ara-Maltsiniotis-2016-joins-and-slices}}]
    \label{lem:inductive-globe}
        There is a map $[1]\xlaxs\cC\to \suspension\cC$, natural in $\cC\in \sCatn[(n-1)]$, which is part of a pushout square
         \begin{equation*}
            \begin{tikzcd}
                {\partial[1]\times\cC} & {\partial[1]} \\
                {[1]\xlaxs\cC} & {\suspension\cC}
                \arrow[""{name=0, anchor=center, inner sep=0}, from=1-1, to=1-2]
                \arrow[from=1-1, to=2-1]
                \arrow[from=1-2, to=2-2]
                \arrow[from=2-1, to=2-2]
                \arrow["\lrcorner"{anchor=center, pos=0.125, rotate=180}, draw=none, from=2-2, to=0]
            \end{tikzcd}\qin \sCatn.
        \end{equation*}
    \end{lemma}

    In particular, for $[\vec{b}]\in \deltaCat^{n-1}$, we get a natural map
    \begin{equation*}
        \hat{\pi} \;\colon\; [1] \xlaxs \G[\vec{b}] \;\too\; \suspension \G[\vec{b}] \;\simeq\; \G[1,\vec{b}]\qin \sCatn.
    \end{equation*}
    Writing $[a]=[1]\sqcup_{[0]}\dots\sqcup_{[0]}[1]$, we get an induced map 
    \begin{equation*}
        \hat{\pi} \;\colon\; [a] \xlaxs \G[\vec{b}] \;\too\; \G[a,\vec{b}]\qin \sCatn
    \end{equation*}
    which is also natural in $[a]\in \deltaCat$.

    \begin{definition}
        Consider $\square$ and $\G$ as functors $\deltaCat^n\to \sCatn$.
        The \emph{collapse map} is the natural transformation
        \begin{equation*}
            \pi \;\colon\; \square \;\Too\; \G
        \end{equation*}
        defined inductively on $[\vec{a}]\in \deltaCat^n$ as the composition
        \begin{equation*}
            \square[\vec{a}] \;\simeq\; [a_1] \xlaxs \square[a_2, \dots, a_{n}] \;\xtoo{\id \xlax_s \pi}\; [a_1] \xlaxs \G[a_2,\dots,a_n] \;\xtoo{\hat{\pi}}\; \G[a_1,\dots,a_n].
        \end{equation*}
    \end{definition}  

    Denote the collapse map of the lax $n$-cube by $\pi_n\coloneqq \pi_{[1]^n}\colon \laxcube{n}\to\laxdisk{n}$.

    \begin{lemma}
    \label{lem:collapse-is-retract}
        The map $\laxdisk{n}\to \laxcube{n}$ selecting the top-dimensional cell is a section of
        $\pi_{n}$.
    \end{lemma}

    \begin{proof}
        It is enough to prove that the map $\laxdisk{n}\to [1]\xlaxs \laxdisk{n-1}$ selecting the top-dimensional cell is a section of $\hat{\pi} \colon [1]\xlaxs \laxdisk{n-1}\to \laxdisk{n}$.
        This follows from~\cref{lem:inductive-globe}, by noticing that its restriction $\partial[1]\to \partial[1]\times\laxdisk{n-1}$ is a section of $\partial[1]\times\laxdisk{n-1}\to \partial[1]$.
    \end{proof}

    \begin{corollary}
    \label{cor:pushout-square-and-globe}
        The following are pushout squares in $\Seg(\globe)$:
        \begin{equation*}\begin{tikzcd}
            {\iota\partial\laxdisk{n}} & {\iota\partial\laxcube{n}} && {\iota\partial\laxcube{n}} & {\iota\partial\laxdisk{n}} \\
            {\laxdisk{n}} & {\laxcube{n}} && {\laxcube{n}} & {\laxdisk{n}}
            \arrow[from=1-1, to=1-2]
            \arrow[from=1-1, to=2-1]
            \arrow[from=1-2, to=2-2]
            \arrow["{\iota\partial\pi_{n}}", from=1-4, to=1-5]
            \arrow[from=1-4, to=2-4]
            \arrow[from=1-5, to=2-5]
            \arrow[from=2-1, to=2-2]
            \arrow["\lrcorner"{anchor=center, pos=0.125, rotate=180}, draw=none, from=2-2, to=1-1]
            \arrow["{\pi_{n}}", from=2-4, to=2-5]
            \arrow["\lrcorner"{anchor=center, pos=0.125, rotate=180}, draw=none, from=2-5, to=1-4]
        \end{tikzcd}\end{equation*}
    \end{corollary}

    \begin{proof}
        It is enough to show that the left square and their composition are pushouts.
        The left square follows from the pasting theorem~\cite[Corollary~5.12]{Campion-2023-pasting}, and the composition  is degenerate by~\cref{lem:collapse-is-retract}.
    \end{proof}

\subsection{Maps of lax cubes}
\label{subsec:maps-of-lax-cubes}
In this subsection, we describe the structure maps between lax cubes that we will use throughout the paper. 

    \subsubsection*{Faces and degeneracies} 
    The simplest maps of lax cubes are the \emph{faces} and \emph{degeneracies}, which are induced from maps of simplices. 
    Explicitly, the \emph{basic faces} are
    \begin{equation*}
        \delta^-,\delta^+ \;\colon\; \laxcube{0} \;\too\; \laxcube{1}
    \end{equation*}
    choosing the source and target respectively, and the \emph{basic degeneracy} is the unique map 
    \begin{equation*}
        \sigma \;\colon\; \laxcube{1} \;\too\; \laxcube{0}.
    \end{equation*}

    \begin{definition}
        Let $I\subseteq \{1,\dots,n\}$ be a subset of cardinality $k$ and let $s\colon I\to \{-,+\}$.
        \begin{enumerate}
            \item The \emph{$(I,s)$-face} $
            \delta_I^s\colon \laxcube{n-k}\to \laxcube{n}$ is given by $\delta_I^s=\delta_{I,1}^{s} \xlaxs \dots \xlaxs \delta_{I,n}^{s}$, where
            \begin{equation*}
                \delta_{I,i}^{s} \;=\; \begin{cases}
                    \delta^{s(i)} \;\colon\; \laxcube{0} \;\too\; \laxcube{1}, & i \in I, \\
                    \id \;\colon\; \laxcube{1} \;\too\; \laxcube{1}, & i \notin I.
                \end{cases}
            \end{equation*}
            \item The \emph{I-degeneracy} $\sigma_I \colon \laxcube{n} \to \laxcube{n-k} $ is given by $\sigma_I= \sigma_{I,1} \xlaxs \dots \xlaxs \sigma_{I,n}$, where
            \begin{equation*}
                \sigma_{I,i} \;=\;
                \begin{cases}
                    \sigma \;\colon\; \laxcube{1} \;\too\; \laxcube{0}, & i\in I,\\
                    \id \;\colon\; \laxcube{1} \;\too\; \laxcube{1}, & i\notin I.
                \end{cases}
            \end{equation*}
        \end{enumerate}
    When $I=\{i\}$ is a singleton, we denote
    \begin{equation*}
        \delta_i^{s(i)} \coloneqq \delta_{\{i\}}^{s} \colon \laxcube{n-1} \to \laxcube{n},\qquad \sigma_i \coloneqq \sigma_{\{i\}} \colon \laxcube{n} \to \laxcube{n-1}.
    \end{equation*}
    \end{definition}

    \subsubsection*{Connections}

    The cubes admit additional degeneracies, relating the different directions, called \emph{connections}.
    The \emph{basic} connections are 
    \begin{equation*}
        \gamma^-,\gamma^+ \;\colon\; \laxcube{2} \;\too\; \laxcube{1}
    \end{equation*}
    which are represented by the following diagrams:
    \begin{equation*}
        \gamma^-:
        \begin{tikzcd}
            {-} & {+} \\
            {+} & {+}
            \arrow[from=1-1, to=1-2]
            \arrow[from=1-1, to=2-1]
            \arrow[equals, from=1-2, to=2-2]
            \arrow[equals, from=2-1, to=2-2]
        \end{tikzcd}\;\too\; \begin{tikzcd}
            {-} \\
            {+,}
            \arrow[from=1-1, to=2-1]
        \end{tikzcd}
        \qquad
        \gamma^+:
        \begin{tikzcd}
            {-} & {-} \\
            {-} & {+}
            \arrow[equals, from=1-1, to=1-2]
            \arrow[equals, from=1-1, to=2-1]
            \arrow[from=1-2, to=2-2]
            \arrow[from=2-1, to=2-2]
        \end{tikzcd}\;\too\;
        \begin{tikzcd}
            {-} \\
            {+.}
            \arrow[from=1-1, to=2-1]
        \end{tikzcd}
    \end{equation*}
    The horizontal and vertical compositions of the two basic connections result in vertical and horizontal degeneracies respectively.
    Namely, the compositions
    \begin{equation*}
        \laxcube{2} \;\too\; \laxcube{2} \sqcup_{\laxcube{1}} \laxcube{2} \;\xtoo{\gamma^+ \sqcup_{\id} \gamma^-}\; \laxcube{1} \sqcup_{\laxcube{1}} \laxcube{1} \;\simeq\; \laxcube{1},
    \end{equation*}
    represented by the following diagrams, are degeneracy maps:
    \begin{equation*}
        \begin{tikzcd}
            - && + \\
            - && +
            \arrow[from=1-1, to=1-3]
            \arrow[from=1-1, to=2-1]
            \arrow[from=1-3, to=2-3]
            \arrow[from=2-1, to=2-3]
        \end{tikzcd}
        \;\too\;
        \begin{tikzcd}
            {-} & {-} & {+} \\
            {-} & {+} & {+}
            \arrow[equals, from=1-1, to=1-2]
            \arrow[equals, from=1-1, to=2-1]
            \arrow[from=1-2, to=1-3]
            \arrow[from=1-2, to=2-2]
            \arrow[equals, from=1-3, to=2-3]
            \arrow[from=2-1, to=2-2]
            \arrow[equals, from=2-2, to=2-3]
        \end{tikzcd}
        \;\too\;
        \begin{tikzcd}
            {-} \\
            {+,}
            \arrow[from=1-1, to=2-1]
        \end{tikzcd}
    \end{equation*}
    \begin{equation*}
        \begin{tikzcd}
            - & - \\
            \\
            + & +
            \arrow[from=1-1, to=1-2]
            \arrow[from=1-1, to=3-1]
            \arrow[from=1-2, to=3-2]
            \arrow[from=3-1, to=3-2]
        \end{tikzcd}
        \;\too\;
        \begin{tikzcd}
            {-} & {-} \\
            {-} & {+} \\
            {+} & {+}
            \arrow[equals, from=1-1, to=1-2]
            \arrow[equals, from=1-1, to=2-1]
            \arrow[from=1-2, to=2-2]
            \arrow[from=2-1, to=2-2]
            \arrow[from=2-1, to=3-1]
            \arrow[equals, from=2-2, to=3-2]
            \arrow[equals, from=3-1, to=3-2]
        \end{tikzcd}
        \;\too\;
        \begin{tikzcd}
            {-} & {+.}
            \arrow[from=1-1, to=1-2]
        \end{tikzcd}
    \end{equation*}
    
    For every $1 \le i \le n-1$ and $s \in \{+, -\}$, we define the connection  $\gamma_i^s \colon \laxcube{n} \to \laxcube{n-1}$ by 
    \begin{equation*}
        \gamma^s_{i} \coloneqq \id \xlaxs \gamma^s \xlaxs \id \;\colon\; \laxcube{i-1} \xlaxs \laxcube{2} \xlaxs \laxcube{n-i-1} \;\too\; \laxcube{i-1} \xlaxs \laxcube{1} \xlaxs \laxcube{n-i-1}.
    \end{equation*}

    \subsubsection*{The fold map}
    
    To relate the cubes, which are direction-symmetric, to the direction-asymmetric $n$-fold Segal spaces, we use the connections to ``fold'' all the faces away from the constancy directions.
    For that purpose, we define the \emph{basic fold map}
    $\psi\colon \laxcube{2}\to \laxcube{2}$
    as the composition
    \begin{equation*}
        \laxcube{2} \;\too\; \laxcube{2} \sqcup_{\laxcube{1}} \laxcube{2} \sqcup_{\laxcube{1}} \laxcube{2} \;\xtoo{\gamma^+ \sqcup_{\id} \id \sqcup_{\id} \gamma^-}\; \laxcube{1} \sqcup_{\laxcube{1}} \laxcube{2} \sqcup_{\laxcube{1}} \laxcube{1} \;\simeq\; \laxcube{2},
    \end{equation*}
    \begin{equation*}
        \begin{tikzcd}
            \bullet &&& \bullet \\
            \bullet &&& \bullet
            \arrow[from=1-1, to=1-4]
            \arrow[from=1-1, to=2-1]
            \arrow[from=1-4, to=2-4]
            \arrow[from=2-1, to=2-4]
        \end{tikzcd}
        \;\too\;
        \begin{tikzcd}
            \bullet & \bullet & \bullet & \bullet \\
            \bullet & \bullet & \bullet & \bullet
            \arrow[equals, from=1-1, to=1-2]
            \arrow[equals, from=1-1, to=2-1]
            \arrow[from=1-2, to=1-3]
            \arrow[from=1-2, to=2-2]
            \arrow[from=1-3, to=1-4]
            \arrow[from=1-3, to=2-3]
            \arrow[equals, from=1-4, to=2-4]
            \arrow[from=2-1, to=2-2]
            \arrow[from=2-2, to=2-3]
            \arrow[equals, from=2-3, to=2-4]
        \end{tikzcd}
        \;\too\;
        \begin{tikzcd}
            \bullet & \bullet \\
            \bullet & \bullet.
            \arrow[from=1-1, to=1-2]
            \arrow[from=1-1, to=2-1]
            \arrow[from=1-2, to=2-2]
            \arrow[from=2-1, to=2-2]
        \end{tikzcd}
    \end{equation*}
    For $1\le i\le n-1$, let $\psi_i\colon \laxcube{n}\to \laxcube{n}$ be 
    \begin{equation*}
        \psi_{i} \coloneqq \id \xlaxs \psi \xlaxs \id \;\colon\;
        \laxcube{i-1} \xlaxs \laxcube{2} \xlaxs \laxcube{n-i-1} \;\too\;
        \laxcube{i-1} \xlaxs \laxcube{2} \xlaxs \laxcube{n-i-1}.
    \end{equation*}
    
    \begin{definition}
        The \emph{fold map} $\Phi_{n}\colon \laxcube{n}\to \laxcube{n}$ is defined as the composition
        \begin{equation*}
            \Phi_{n} \;=\; \psi_{n-1}(\psi_{n-2}\psi_{n-1})\cdots(\psi_1\psi_2\cdots\psi_{n-1}).
        \end{equation*}
    \end{definition}

    \begin{remark}
        We use a slightly different definition of $\psi_i$ and $\Phi_{n}$ than in~\cite{Al-Agl-Brown-Steiner-2002-strict-globular-cubical}, amounting to a different choice of constancy directions.
    \end{remark}

    In~\cite{Al-Agl-Brown-Steiner-2002-strict-globular-cubical}, an explicit description of $\Phi_{n}$ is given on the globular cells of $\laxcube{n}$.
    Every cell in $\laxcube{n}$ is either degenerate or a composite of \emph{elementary cells}, 
    which are the maximal cells of the various faces.
    \begin{definition}
        Given $I\subseteq \{1,\dots,n\}$ and $s\colon I\to \{-,+\}$, the \emph{elementary $(I,s)$-cell} of $\laxcube{n}$ is the maximal cell of the $(I,s)$-face: 
        \begin{equation*}
            \laxdisk{n-|I|} \;\too\; \laxcube{n-|I|} \;\xtoo{\delta_I^s}\; \laxcube{n}.
        \end{equation*}
        When $I=\emptyset$, the corresponding elementary cell is the top-dimensional cell of $\laxcube{n}$.
    \end{definition}

    By~\cite[Proposition~3.2]{Al-Agl-Brown-Steiner-2002-strict-globular-cubical}, the fold map $\Phi_n\colon \laxcube{n}\to \laxcube{n}$ acts as follows on the elementary $(I,s)$-cell:
        \begin{itemize}
            \item The top-dimensional cell ($I = \emptyset$) is sent to itself.
            \item If $I\neq \emptyset$, let $m=\min(I)$. 
            The elementary $(I,s)$-cell is sent to the $(m-1)$-dimensional source (if $s(m)=-$) or target (if $s(m)=+)$ of the top-dimensional cell
            \begin{equation*}
                \laxdisk{m-1} \;\too\; \laxdisk{n} \;\too\; \laxcube{n}.
            \end{equation*}
        \end{itemize}
    
    \begin{proposition}
    \label{lem:endomorphism-equal-fold-map}
        The idempotent $\laxcube{n}\xto{\pi_{n}} \laxdisk{n}\to \laxcube{n}$ is equal to the fold map $\Phi_{n}$.
    \end{proposition}

    \begin{proof}
        It is enough to check that the idempotent agrees with $\Phi_n$ on the elementary cells, as described above.
        For the top-dimensional cell, it follows from $\pi_{n}$ being a retract (\cref{lem:collapse-is-retract}).
        Let $\emptyset\ne I\subseteq\{1,\dots, n\}$ of cardinality $k$ and $s\colon I\to \{-,+\}$, and denote $m=\min(I)$.
        We need to show that the composition
        \begin{equation*}
            \laxdisk{n-k} \;\too\; \laxcube{n-k} \;\xtoo{\delta_I^s}\; \laxcube{n} \;\xtoo{\pi_n}\; \laxdisk{n}
        \end{equation*}
        factors as the degeneracy $\laxdisk{n-k}\to \laxdisk{m-1}$ followed by the inclusion to the $(m-1)$-dimensional source or target.

        First, assume that $m=1$. 
        In particular, the $(I,s)$-face factors through $\partial[1]\times\laxcube{n-1}\to \laxcube{n}$.
        Recall that $\pi_{n}$ is defined as 
        \begin{equation*}
            \laxcube{n} \;\simeq\; [1] \xlaxs \laxcube{n-1} \;\xtoo{\id\xlaxs \pi_{n-1}}\;  [1] \xlaxs \laxdisk{n-1} \;\xtoo{\hat{\pi}}\; \laxdisk{n},
        \end{equation*}
        and $\hat{\pi}$ collapses the faces $\partial[1]\times \laxdisk{n-1}$ to $\partial[1]$ (\cref{lem:inductive-globe}).
        In particular, the $(I,s)$-face is sent to the 0-dimensional source or target, as needed.

        If $m>1$, denote 
        \begin{equation*}
            \tilde{I} \;\coloneqq\; I-(m-1) \;\subseteq\; \{1,\dots, n-m+1\},
        \end{equation*}
        and
        \begin{align*} 
            \tilde{s} \;\colon\; \tilde{I} &\;\too\; \{+,-\} \\
            i & \;\mapsto\;  s(i+m-1).
        \end{align*}
        We can decompose $\delta_I^s$ as
        \begin{equation*}
            \id \xlaxs \delta_{\tilde{I}}^{\tilde{s}} \;\colon\; {\laxcube{m-1} \xlaxs \laxcube{n-k-m+1}} \;\too\; {\laxcube{m-1} \xlaxs \laxcube{n-m+1}}.
        \end{equation*}
        Moreover, $\pi_{n}$ factors as 
        \begin{equation*}
            \laxcube{n} \;\simeq\; {\laxcube{m-1} \xlaxs \laxcube{n-m+1}} \;\xtoo{\id\xlaxs \pi_{n-m+1}}\; {\laxcube{m-1} \xlaxs \laxdisk{n-m+1}} \;\too\; {\laxdisk{n}}.
        \end{equation*}
        Thus, the result follows from the commuting diagram
        \begin{equation*}\begin{tikzcd}
            {\laxdisk{n-k}} & {\laxcube{m-1}\xlaxs \laxdisk{n-k-m+1}} & {\laxcube{m-1}\xlaxs \laxcube{n-k-m+1}} \\
            && {\laxcube{m-1}\xlaxs \laxcube{n-m+1}} \\
            {\laxdisk{m-1}} & {\laxcube{m-1}\xlaxs \pt} & {\laxcube{m-1}\xlaxs \laxdisk{n-m+1}} \\
            & {\laxdisk{m-1}} & {\laxdisk{n}}
            \arrow[from=1-1, to=1-2]
            \arrow[from=1-1, to=3-1]
            \arrow[from=1-2, to=1-3]
            \arrow[from=1-2, to=3-2]
            \arrow["{\id\xlaxs \delta_{\tilde{I}}^{\tilde{s}}}", from=1-3, to=2-3]
            \arrow["{\id\xlaxs\pi_{n-m+1}}", from=2-3, to=3-3]
            \arrow[from=3-1, to=3-2]
            \arrow[equals, from=3-1, to=4-2]
            \arrow[from=3-2, to=3-3]
            \arrow["{\pi_{m-1}}", from=3-2, to=4-2]
            \arrow[from=3-3, to=4-3]
            \arrow[from=4-2, to=4-3]
        \end{tikzcd}\end{equation*}
        where the top row selects the top-dimensional cell, the top-right square commutes from the case $m=1$, the bottom-left triangle commutes by~\cref{lem:collapse-is-retract}, and the rest commute by naturality.            
    \end{proof}  

\subsection{$\grid$-Spaces}
\label{subsec:grid-spaces}

    Recall that the $n$-grid category $\grid$ was defined as the essential image of $\square\colon \deltaCat^n\to \sCatn$.
    We now define a geometric pattern on $\grid$.

    \begin{definition} 
        We define a geometric pattern on $\grid$ as follows:
        \begin{enumerate}
            \item Inert morphisms $\varphi \colon \square[\vec{a}] \to \square[\vec{b}]$ are those of the form $\varphi = \square(f)$, where $f \colon [\vec{a}] \to [\vec{b}]$ is an inert morphism in $\deltaCat^{n}$.
            \item Active morphisms $\varphi \colon \square[\vec{a}] \to \square[\vec{b}]$ are those that preserve the minimal and maximal corners. Namely, $\varphi(\vec{0}) = \vec{0}$ and $\varphi(\vec{a}) = \vec{b}$.
            \item The elementary objects are the lax cubes $\laxcube{0}, \dots, \laxcube{n}$.
        \end{enumerate}
    \end{definition}

    \begin{proposition}
        The above defines a geometric pattern on $\grid$.
    \end{proposition}

    \begin{proof}
        We need to check that the inert and active morphisms form a factorization system.

        Let $\varphi \colon \square[\vec{a}] \to \square[\vec{b}]$ be a morphism in $\grid \subseteq \sCatn$. Let
        \begin{equation*}  
            (m_1, \dots ,m_{n}) \;\coloneqq\; \varphi(\vec{0}), \qquad (M_1, \dots, M_{n}) \;\coloneqq\; \varphi(\vec{a}),
        \end{equation*}
        Note that $m_i\le M_i$ (\cref{rem:maps-of-grids-order}), and define $c_i \coloneqq M_i - m_i \ge 0$.
        Define the simplicial maps
        \begin{align*}
            & g \colon [\vec{M}] \too [\vec{c}], && h \colon [\vec{c}] \too [\vec{b}] && \in && \deltaCat^{n}.\\
            & g(\pvec{i}) = \begin{cases}
                \vec{0}, & \vec{m} \not\le \pvec{i} \\
                \pvec{i} - \vec{m}, & \vec{m} \le \pvec{i}
            \end{cases} && h(\pvec{i}) = \vec{m} + \pvec{i}
        \end{align*}
        Noting that $\varphi$ factors through the embedding $\square[\vec{M}] \into \square[\vec{b}]$, we define the functors
        \begin{align*}
            & \psi \coloneqq \square[g] \circ \varphi \colon \square[\vec{a}] \too \square[\vec{c}], && \xi \coloneqq \square[h] \colon \square[\vec{c}] \too \square[\vec{b}].
        \end{align*}
        Then $\psi$ is active and $\xi$ is inert. To prove that $\varphi = \xi \circ \psi$, it is enough to notice that $\varphi(\pvec{i}) \ge \vec{m}$ for all $\pvec{i} \in \square[\vec{a}]$, so that on its image, $h$ is left inverse to $g$.

        Assume we have  another such decomposition into active and inert morphisms
        \begin{equation*}
            \square[\vec{a}] \;\xtoo{\psi'}\; \square[\pvec{c}'] \;\xtoo{\xi'}\; \square[\vec{b}].
        \end{equation*}
        As $\xi'$ is inert, it is of the form $\xi' = \square[h']$ and $h'(\vec{i}) = h'(\vec{0}) + \vec{i}$.
        Since $\psi'$ is active, $\psi'(\vec{0}) = \vec{0}$, so 
        \begin{equation*}
            \vec{m} \;=\; \varphi(\vec{0}) \;=\;\xi'(\psi'(\vec{0})) \;=\;\xi'(\vec{0}) \;=\; h'(\vec{0}).
        \end{equation*}
        Again using that $\psi'$ is active, $\psi'(\vec{a}) = \pvec{c}'$, so we get
        \begin{equation*}
            \vec{M} \;=\; \varphi(\vec{a}) \;=\;\xi'(\psi'(\vec{a})) \;=\;\xi'(\pvec{c}') \;=\; h'(\pvec{c}') \;=\; \vec{m} + \pvec{c}'.
        \end{equation*}
        So $\pvec{c}' = \vec{c}$ and $\xi' = \xi$.
        The result now follows, as $\xi$ is a monomorphism and
        \begin{equation*}
            \xi \circ \psi \;=\; \varphi \;=\;\xi \circ \psi'.
        \end{equation*}
    \end{proof}

    The functor $\square\colon\deltaCat^n\to \grid$ is a morphism of geometric patterns: inert and active maps are preserved by definition, and the elementary object $[1]^I\in \deltaCat^n$ is mapped to the elementary object $\laxcube{|I|}\in \grid$.

    \begin{lemma}
    \label{lem:elementary-in-grid-and-delta}
        For every $[\vec{a}]\in \deltaCat^n$, the functor 
        \begin{equation*}
            (\deltaCat^n)^\el_{/[\vec{a}]} \;\too\; (\grid)^\el_{/\square[\vec{a}]},
        \end{equation*}
        induced from $\square\colon \deltaCat^n\to \grid$, is an equivalence of categories.
    \end{lemma}

    \begin{proof}
        This will follow from the stronger claim that $\square\colon \deltaCat^n\to \grid$ admits unique liftings of inert maps.
        Namely, for every $Q\in\grid$ and an inert $\varphi\colon Q\to \square[\vec{a}]$, there exists a unique inert $f\colon [\vec{b}]\to [\vec{a}]$ such that $\square(f)=\varphi$.

        Existence follows from the definition of inert maps in $\grid$.
        For uniqueness, suppose that $f\colon [\vec{b}]\to[\vec{a}]$ lifts $\varphi$.
        Denote by
        $0\le \vec{m}\le \vec{M}\le \vec{a}$ the images under $\varphi$ of the minimal and maximal corners of $Q$ respectively.
        In particular, we have
        \begin{equation*}
            f(\vec{0})=\vec{m},\qquad f(\vec{b})=\vec{M}.
        \end{equation*}
        However, since $f$ is inert, we also have $f(\vec{b})=\vec{m}+\vec{b}$, thus $\vec{b}=\vec{M}-\vec{m}$ and $[\vec{b}]$ is unique.
        Moreover, for all $0\le \vec{c}\le \vec{a}$, $f(\vec{c})=\vec{m}+\vec{c}$,
        hence $f$ is unique.
    \end{proof}
    
    \begin{corollary}            
        The functor $\square\colon \deltaCat^{n}\to \grid$ is a strong Segal morphism of geometric patterns.
    \end{corollary}

    \begin{lemma}            
        The geometric pattern $\grid$ is saturated.
    \end{lemma}

    \begin{proof}
        Note that, for every $[a]\in \deltaCat$, we have
        \begin{equation*}
            \colim_{[i]\in \deltaCat^\el_{/[a]}} [i] \;\isotoo\; [a] \qin \sCat.
        \end{equation*}
        As the strict Gray product commutes with colimits in each variable, we get that for every $[\vec{a}] \in \deltaCat^n$
        \begin{equation*}
            \colim_{[1]^I \in (\deltaCat^n)^\el_{/[\vec{a}]}} \laxcube{|I|} \;\isotoo\; \square[\vec{a}] \qin \sCatn.
        \end{equation*}
        The result then follows from~\cref{lem:elementary-in-grid-and-delta}.
    \end{proof}

    \begin{lemma}
        The lax $n$-cube $\laxcube{n}$ is a weak generator of $\Shv(\grid)$.
    \end{lemma}
    \begin{proof}
        By the Segal condition, $\Shv(\grid)$ is generated under colimits by the elementary objects $\laxcube{0}, \dots, \laxcube{n}$. Since every $\laxcube{k}$ is a retract of $\laxcube{n}$ for $k \le n$, the single object $\laxcube{n}$ suffices to generate the category.
    \end{proof}

    \subsubsection*{Boundary} 

    The geometric pattern $\grid$ has a maximum elementary $\laxcube{n}$, and the boundary inclusion
    $(\grid)_\partial \into \grid$ agrees with the inclusion $\grid[n-1]\into \grid$.
    In particular, this inclusion is a strong Segal morphism by~\cref{lem:boundary-inclusion-is-strong-segal}.
    Consider the adjunction as in~\cref{def:boundary-and-inclusion}
    \begin{equation*}
        \iota \;\colon\; \Shv(\grid[n-1];\cC) \;\longadj\; \Shv(\grid;\cC) \;\cocolon\; \partial.
    \end{equation*}

    \begin{remark}
    \label{rem:inclusion-of-grids}
        The inclusion $\grid[n-1]\into\grid$ does not satisfy the assumptions of~\cref{prop:core-is-localization}.
        For example, there are active maps $\laxcube{1}\to \laxcube{n}$ given by spine compositions.
        Nevertheless, we will eventually learn that  $\iota\colon \Shv(\grid[n-1])\to\Shv(\grid)$ is fully faithful, as by~\cref{thm:equivalence-n} it is equivalent to $\iota\colon \Shv(\globe[n-1])\into \Shv(\globe)$.
    \end{remark}
    
    Consider the following commuting square of geometric patterns, and the corresponding commuting square of sheaves with the $(-)^*$-functoriality:
    \begin{equation*}
        \begin{tikzcd}
            {\deltaCat^n_\partial} & {\grid[n-1]} \\
            {\deltaCat^n} & \grid
            \arrow["\square", from=1-1, to=1-2]
            \arrow[hook, from=1-1, to=2-1]
            \arrow[hook, from=1-2, to=2-2]
            \arrow["\square", from=2-1, to=2-2]
        \end{tikzcd}
        \quad\mapsto\quad 
        \begin{tikzcd}
            {\Shv(\deltaCat^n_\partial)} & {\Shv(\grid[n-1]).} \\
            {\Shv(\deltaCat^n)} & {\Shv(\grid).}
            \arrow["{{{{{\square^*}}}}}"', from=1-2, to=1-1]
            \arrow["\partial", from=2-1, to=1-1]
            \arrow["\partial"', from=2-2, to=1-2]
            \arrow["{{{{{\square^*}}}}}"', from=2-2, to=2-1]
        \end{tikzcd}
    \end{equation*}
    
    The latter square satisfies the horizontal Beck--Chevalley condition on representables.
    \begin{lemma}
    \label{lem:square-bc-on-representables}
        Let $[\vec{a}]\in \deltaCat^n$, then the Beck--Chevalley comparison map
        \begin{equation*}
            \square_! \partial [\vec{a}] \;\too\; \partial\square_![\vec{a}] \qin \Shv(\grid[n-1])
        \end{equation*}
        is an isomorphism.
    \end{lemma}

    \begin{proof}
        By the Segal condition, 
        \begin{equation*}
                \partial[\vec{a}] \;\simeq\; \colim_{[1]^I \in (\deltaCat^n)^\el_{/[\vec{a}]},\ |I|<n} [1]^I,\qquad
                \partial \square[\vec{a}] \;\simeq\; \colim_{\laxcube{k} \in (\grid)^\el_{/\square[\vec{a}]},\ k<n} \laxcube{k},
            \end{equation*}
            and $\square_!$ commutes with colimits and sends $[1]^I\mapsto \laxcube{|I|}$. 
            It follows that $ \square_!\partial [\vec{a}]\to\partial\square_![\vec{a}]$ is an isomorphism by~\cref{lem:elementary-in-grid-and-delta}.
    \end{proof}
    Note that the inclusions
    \begin{equation*}
        \emptyset = \grid[-1] \;\intoo\; \grid[0] \;\intoo\; \grid[1] \;\intoo\; \grid[2] \;\intoo\; \cdots
    \end{equation*}
    are all strong Segal morphisms, so by~\cref{cor:gpat-filtered-colimit-strong-segal} they have a colimit in $\GPatSeg$.
    \begin{definition}
        Define the geometric pattern
        \begin{equation*}
            \grid[] \;=\; \grid[\infty] \;\coloneqq\; \colim_{n} \grid[n] \qin \GPatSeg.
        \end{equation*}
    \end{definition}
    Alternatively, $\grid[]$ is the full subcategory of $\sCatoo$ on all lax grids.
   \Cref{cor:gpat-filtered-colimit-strong-segal} also implies that
    \begin{equation*}
        \colim_{n} \Shv(\grid) \;\isotoo\; \Shv(\grid[]) \qin \PrL,
    \end{equation*}
    or equivalently, by passing to right adjoints, 
    \begin{equation*}
        \Shv(\grid[];\cC) \;\isotoo\; \lim_n \Shv(\grid;\cC) \qin \CAT.
    \end{equation*}
    
    \begin{remark}
    \label{rem:grid-sets-are-cubical-categories-with-connections}
        In~\cite{Al-Agl-Brown-Steiner-2002-strict-globular-cubical}, a model for strict $\infty$-categories was given in terms of \emph{strict cubical $\infty$-categories with connections}.
        A sheaf of sets $X\in\Shv(\grid[];\Set)$ has a canonical structure of this kind:
        faces, degeneracies, and connections are induced from the corresponding maps of cubes, and composition is determined by the Segal condition.
    \end{remark}

    \subsubsection*{Univalence}

    \begin{definition}
        Let $X \in \Shv(\grid)$. We say that $X$ is \emph{univalent} if $\square^* X\in \Shv(\deltaCat^n)$ is a univalent $n$-uple Segal space.
        We denote the full subcategory of univalent sheaves by $\Shv(\grid)\univ$.
    \end{definition}

    The boundary sends univalent sheaves to univalent sheaves and therefore gives rise to a functor
    \begin{equation*}
        \partial \;\colon\; \Shv(\grid)\univ \;\too\; \Shv(\grid[n-1])\univ.
    \end{equation*}
    Taking the limit, we define
    \begin{equation*}
        \Shv(\grid[])\univ \;\coloneq\; \lim_n \Shv(\grid)\univ.
    \end{equation*}

    \subsection{The Gray product} 

    The subcategory $\grid[]\subseteq\sCatoo$ is closed under the strict Gray product, so it inherits a monoidal structure.
    The following lemma implies that this monoidal structure lifts to a monoid structure on $\grid[]$ in $\GPatSeg$.

    \begin{lemma}
        For every $0\le n,k\le \infty$,
        the strict Gray product 
        \begin{equation*}
            \xlaxs \;\colon\; \grid[n]\times\grid[k] \;\too\; \grid[n+k]
        \end{equation*}
        is a Segal morphism of geometric patterns.
    \end{lemma}

    \begin{proof}
        It is enough to prove for $n,k<\infty$, the infinite case follows by taking colimits.
        Let $\square[\vec{a}]\in \grid[n]$ and $\square[\vec{b}]\in \grid[k]$.
        By~\cref{lem:elementary-in-grid-and-delta}, we have equivalences
        \begin{equation*}
            (\grid[n]\times\grid[k])^\el_{/(\square[\vec{a}],\square[\vec{b}])} 
            \;\simeq\; (\grid[n])^\el_{/\square[\vec{a}]}\times (\grid[k])^\el_{/\square[\vec{b}]} 
            \;\simeq\;  (\deltaCat^n)^\el_{/[\vec{a}]}\times (\deltaCat^k)^\el_{/[\vec{b}]} 
            \;\simeq\; (\deltaCat^{n+k})^\el_{/[\vec{a},\vec{b}]}
            \;\simeq\; (\grid[n+k])^\el_{/\square[\vec{a},\vec{b}]}.
        \end{equation*}
        The above composition agrees with the map induced from the strict Gray product. 
    \end{proof}

    \begin{corollary}
    \label{cor:gray-product}
        The category $\Shv(\grid[])$ admits a presentably monoidal structure given by Day convolution followed by Segalification. 
    \end{corollary}

    \begin{proof}
        The functor $\Shv \colon \GPatSeg \to \PrL$ is symmetric monoidal by \cref{lem:sheavs-are-symmetric-monoidal}. In particular, it sends the monoid object $\grid[]$ in $\GPatSeg$ to a presentably monoidal category.
    \end{proof}
    
    \begin{definition}
    \label{def:Gray-product}
        The resulting monoidal structure on $\Shv(\grid[])$ is called the \emph{Gray product} and is denoted by $\xlax$.
        Explicitly, the binary operation is given by the composition
        \begin{equation*}
            \xlax \;\colon\; \Shv(\grid[]) \otimes \Shv(\grid[]) \;\simeq\; \Shv(\grid[] \times \grid[]) \;\xtoo{(\xlaxs)_!}\; \Shv(\grid[]).
        \end{equation*}
    \end{definition}
    
    Similarly, we use the same symbol to denote the composition
    \begin{equation*}
        \xlax \;\colon\; \Shv(\grid[n]) \otimes \Shv(\grid[k]) \;\simeq\; \Shv(\grid[n] \times \grid[k]) \;\xtoo{(\xlaxs)_!}\; \Shv(\grid[n+k]).
    \end{equation*}

    \begin{example}
    \label{exm:gray-is-strict-on-rep}
        The Gray product agrees with the strict Gray product on representables,
        \begin{equation*}
            \square[\vec{a}] \xlax \square[\vec{b}] \;\simeq\; \square[\vec{a}] \xlaxs \square[\vec{b}] \;\simeq\; \square[\vec{a},\vec{b}]\;\simeq\; \square([\vec{a}]\boxtimes[\vec{b}]).
        \end{equation*} 
        More generally, the functor $\square_!$ sends external products to Gray products.
        This can be seen by applying $\Shv\colon \GPatSeg\to \PrL$ to the following commuting square:
        \begin{equation*}
            \begin{tikzcd}
                {\deltaCat^n\times \deltaCat^k} & {\deltaCat^{n+k}} \\
                {\grid[n]\times\grid[k]} & {\grid[n+k].}
                \arrow["\sim", from=1-1, to=1-2]
                \arrow["{\square\times \square}"', from=1-1, to=2-1]
                \arrow["\square", from=1-2, to=2-2]
                \arrow["\xlaxs", from=2-1, to=2-2]
            \end{tikzcd}
        \end{equation*}
    \end{example}

    \begin{remark}
        Tensoring with $\Set$, we see that $\Shv(\grid[];\Set)$ has a monoidal structure
        \begin{equation*}
            \pi_0(- \xlax - ) \;\colon\; \Shv(\grid[];\Set) \otimes \Shv(\grid[];\Set) \;\too\; \Shv(\grid[];\Set)
        \end{equation*}
        given by the Gray product followed by $\pi_0$.
        This agrees with the strict Gray product defined in~\cite{Al-Agl-Brown-Steiner-2002-strict-globular-cubical} on strict cubical $\infty$-categories with connections (\cref{rem:grid-sets-are-cubical-categories-with-connections}).
    \end{remark}

\section{Comparing lax grids and globular sums}
    
Consider the following adjunctions:
\begin{equation*}\begin{tikzcd}[column sep=large]
    {\Shv(\globe )} & {\Shv(\deltaCat^{n})} & {\Shv(\grid).}
    \arrow[""{name=0, anchor=center, inner sep=0}, "{{{\G^*}}}", from=1-1, to=1-2]
    \arrow[""{name=1, anchor=center, inner sep=0}, "{{{\G_*}}}", shift left=3, from=1-2, to=1-1]
    \arrow[""{name=2, anchor=center, inner sep=0}, "{{{\square_!}}}", shift left=3, from=1-2, to=1-3]
    \arrow[""{name=3, anchor=center, inner sep=0}, "{{{\square^*}}}", from=1-3, to=1-2]
    \arrow["\dashv"{anchor=center, rotate=-90}, draw=none, from=0, to=1]
    \arrow["\dashv"{anchor=center, rotate=-90}, draw=none, from=2, to=3]
\end{tikzcd}\end{equation*}

\begin{definition}
    Denote the composition of the above adjoints by
    \begin{equation*}
        \Fold^\L \coloneqq \square_!\G^* \;\colon\; \Shv(\globe) \;\longadj\; \Shv(\grid) \;\cocolon\; \G_* \square^* \eqqcolon \Fold.
    \end{equation*}
\end{definition}

The goal of this section is to prove that $\Fold$ is an equivalence of categories.
In particular, the inverse of $\Fold$ will be its left adjoint $\Fold^\L$.
We will also give an explicit description of the inverse as the \emph{cubical nerve}.

Our proof strategy is inductive.
Consider the commuting diagram 
\begin{equation*}\begin{tikzcd}
    \grid & {\deltaCat^n} & \globe \\
    {\grid[n-1]} & {\deltaCat^{n-1}} & {\globe[n-1].}
    \arrow["\square"', from=1-2, to=1-1]
    \arrow["\G",from=1-2, to=1-3]
    \arrow[hook, from=2-1, to=1-1]
    \arrow["j", hook, from=2-2, to=1-2]
    \arrow["\square"', from=2-2, to=2-1]
    \arrow["\G",from=2-2, to=2-3]
    \arrow[hook, from=2-3, to=1-3]
\end{tikzcd}\end{equation*}
Taking sheaves with the $(-)^*$-functoriality, and using the Beck--Chevalley condition from~\cref{lem:globe-commute-boundary}, we see that $\Fold=\G_*\square^*$ commutes with $\partial$.
Passing to left adjoints, it also follows that $\Fold^\L$ commutes with $\iota$.

\subsection{The cubical nerve}
    \begin{definition}
        Define the \emph{cubical nerve}
        \begin{equation*}
            \cubicalNerve \;\colon\; \Shv(\globe) \;\too\; \PSh(\grid).
        \end{equation*}
        as the restricted Yoneda functor corresponding to the full subcategory $\grid \subseteq \Shv(\globe)$.
    \end{definition}
    Explicitly, for $X \in \Shv(\globe)$ and $Q \in \grid\op$,
    \begin{equation*}
        \cubicalNerve (X)_Q \;=\; \Map_{\Shv(\globe)}(Q, X).
    \end{equation*}

    \begin{lemma}
    \label{lem:boundary-of-grid-in-fold}
        For every $Q\in\grid$ and $0\le r\le n$, the map
         \begin{equation*}
            \colim_{\laxcube{k}\in (\grid[r]^{\el})_{/Q}} \laxcube{k} \;\too\; \partial^{n-r}Q \qin \Shv(\globe[r]).
        \end{equation*}
        is an isomorphism, where $(\grid[r]^{\el})_{/Q} \coloneqq \grid[r]^{\el} \times_{\grid^{\el}} (\grid^{\el})_{/Q}$.
    \end{lemma}

    \begin{proof}
        We will prove this by induction on $r$, starting with $r=-1$, where both sides agree with $\pt$.
        The above colimit can be computed as a pushout
        \begin{equation*}
            \begin{tikzcd}
                {\bigsqcup\limits_{\laxcube{k} \in (\grid[r]^{\el})_{/Q}} \iota\partial \laxcube{k}} & {\bigsqcup\limits_{\laxcube{k} \in (\grid[r]^{\el})_{/Q}} \laxcube{k}} \\
                {\colim\limits_{\laxcube{k} \in (\grid[r-1]^{\el})_{/Q}} \iota\laxcube{k}} & {\colim\limits_{\laxcube{k} \in (\grid[r]^{\el})_{/Q}} \laxcube{k}}
                \arrow[from=1-1, to=1-2]
                \arrow[from=1-1, to=2-1]
                \arrow[from=1-2, to=2-2]
                \arrow[from=2-1, to=2-2]
                \arrow["\lrcorner"{anchor=center, pos=0.125, rotate=180}, draw=none, from=2-2, to=1-1]
            \end{tikzcd} \qin \Shv(\globe[r]),
        \end{equation*}
        which by induction is
        \begin{equation*}
            \begin{tikzcd}
                {\bigsqcup\limits_{\laxcube{k} \in (\grid[r]^{\el})_{/Q}} \iota\partial \laxcube{k}} & {\bigsqcup\limits_{\laxcube{k} \in (\grid[r]^{\el})_{/Q}} \laxcube{k}} \\
                {\iota \partial^{n-r+1} Q} & {\colim\limits_{\laxcube{k} \in (\grid[r]^{\el})_{/Q}} \laxcube{k}}
                \arrow[from=1-1, to=1-2]
                \arrow[from=1-1, to=2-1]
                \arrow[from=1-2, to=2-2]
                \arrow[from=2-1, to=2-2]
                \arrow["\lrcorner"{anchor=center, pos=0.125, rotate=180}, draw=none, from=2-2, to=1-1]
            \end{tikzcd} \qin \Shv(\globe[r]).
        \end{equation*}
        $\partial^{n-r} Q$ is the \emph{strict} pushout of the above cospan, i.e.\ the pushout in $\Shv(\globe[r];\Set)$. Hence, the claim follows from the pasting theorem~\cite[Corollary~5.12]{Campion-2023-pasting}.
    \end{proof}
    
    \begin{corollary}
    \label{cor:cubical-nerve-is-segal}
        The cubical nerve lands in Segal sheaves 
        \begin{equation*}
            \cubicalNerve \;\colon\; \Shv(\globe) \;\too\; \Shv(\grid).
        \end{equation*}
    \end{corollary}

    \begin{proof}
        Equivalently, we need to show that for every $Q \in \grid$
        \begin{equation*}
            \colim_{\laxcube{k} \in (\grid^{\el})_{/Q}} \laxcube{k} \;\isotoo\; Q \qin \Shv(\globe),
        \end{equation*}
        which follows from~\cref{lem:boundary-of-grid-in-fold}.
    \end{proof}

    \begin{remark}
        The cubical nerve commutes with $\partial$, as realized for $X\in \Shv(\globe)$ and $Q\in \grid[n-1]\op$ by the isomorphism
        \begin{equation*}
            (\partial \cubicalNerve X)_Q
            \;\simeq\;
            \Map_{\Shv(\globe)}(Q,X)
            \;\simeq\;
            \Map_{\Shv(\globe[n-1])}(Q,\partial X)
            \;\simeq\;
            (\cubicalNerve\partial X)_Q.
        \end{equation*}
    \end{remark}

    Pre-composition with the collapse map $\pi_{[\vec{a}]} \colon \square[\vec{a}] \to \G[\vec{a}]$
    induces a map
    \begin{equation*}
         - \circ \pi_{[\vec{a}]} \;\colon\; \Map_{\Shv(\globe)}(\G[\vec{a}], X) \;\too\; \Map_{\Shv(\globe)}(\square[\vec{a}], X)
    \end{equation*}
    natural in $[\vec{a}] \in (\deltaCat^n)\op$ and $X \in \Shv(\globe)$. 
    This corresponds to a natural transformation \mbox{$\G^* \To \square^*\cubicalNerve$}, which by adjunction provides a natural transformation
    
    \begin{equation*}
        \alpha \;\colon\; \id_{\Shv(\globe)} \;\Too\; \Fold\cubicalNerve.
    \end{equation*}
    
    \begin{proposition}
    \label{prop:unit-is-iso}
        The above natural transformation is an isomorphism
        \begin{equation*}
            \alpha \;\colon\; \id_{\Shv(\globe)} \;\isoToo\; \Fold\cubicalNerve.
        \end{equation*}
    \end{proposition}

    \begin{proof}
        We will prove this by induction, the case $n=0$ being trivial.
        Let $X\in \Seg(\globe)$.
        To prove that $\alpha_X\colon X\to \Fold\cubicalNerve X$ is an isomorphism, it is enough to verify after mapping from the weak generator $\laxdisk{n}$:
        \begin{align*}
            \Map_{\Shv(\globe)}(\laxdisk{n}, X)\;\too\; & \Map_{\Shv(\globe)}(\laxdisk{n},\Fold\cubicalNerve X)\\
            \simeq\; & \Map_{\Shv(\deltaCat^n)}(\G^*\laxdisk{n},\square^*\cubicalNerve X)
        \end{align*}
        By~\cref{lem:n-uple-globe-pushout}, the last mapping space is given by the pullback
        \begin{equation*}
            \Map_{\Shv(\deltaCat^n)}(\G^*\iota\partial\laxdisk{n},\square^*\cubicalNerve X)\times_{\Map_{\Shv(\deltaCat^n)}(\iota\partial[1]^n ,\square^*\cubicalNerve X)}\Map_{\Shv(\deltaCat^n)}([1]^n ,\square^*\cubicalNerve X).
        \end{equation*}
        
        We consider each part of the above pullback separately:
        \paragraph{(1) $\G^*\iota\partial\laxdisk{n}$.} 
        We have isomorphisms
        \begin{align*}
             \Map_{\Shv(\deltaCat^n)}(\G^*\iota\partial\laxdisk{n},\square^*\cubicalNerve X) 
             \;\simeq\;&
             \Map_{\Shv(\globe[n-1])}(\partial \laxdisk{n},\partial \Fold\cubicalNerve X)\\
             \;\simeq\;&
             \Map_{\Shv(\globe[n-1])}(\partial \laxdisk{n}, \Fold\cubicalNerve\partial X).
        \end{align*}
        where the last isomorphism uses the fact that $\Fold$ and $\cubicalNerve$ commute with $\partial$.
        By induction, the map $\alpha_{\partial X}\colon \partial X\to \Fold\cubicalNerve \partial X$ is an isomorphism, so we get 
        \begin{align*}
             \Map_{\Shv(\globe[n-1])}(\partial \laxdisk{n}, \Fold\cubicalNerve\partial X) 
             \;\simeq\;&
             \Map_{\Shv(\globe[n-1])}(\partial \laxdisk{n},\partial X)\\
             \;\simeq\;&
             \Map_{\Shv(\globe)}(\iota\partial \laxdisk{n}, X).
        \end{align*}

        The resulting natural map
        \begin{equation*}
            \Map_{\globe}(\laxdisk{n}, X) \;\too\; \Map_{\globe}(\iota \partial \laxdisk{n}, X)
        \end{equation*}
        is given by the composition along the top and right arrows, followed by the inverse of the bottom arrow, in the following commutative square:
        \begin{equation*}
            \begin{tikzcd}
                { \Map_{\Shv(\globe)}(\laxdisk{n}, X)} & {\Map_{\Shv(\deltaCat^n)}(\G^*\laxdisk{n},\square^*\cubicalNerve X)} \\
                { \Map_{\Shv(\globe)}(\iota\partial\laxdisk{n}, X)} & {\Map_{\Shv(\deltaCat^n)}(\G^*\iota\partial\laxdisk{n},\square^*\cubicalNerve X)}
                \arrow[from=1-1, to=1-2]
                \arrow[from=1-1, to=2-1]
                \arrow[from=1-2, to=2-2]
                \arrow["\sim", from=2-1, to=2-2]
            \end{tikzcd}
        \end{equation*}
        where the vertical maps are induced from the counit of the adjunction $\iota \dashv \partial$. Consequently, the map $\iota\partial\laxdisk{n} \to \laxdisk{n}$ corresponding to this transformation is the counit itself.

        \paragraph{(2) $[1]^n$.} 
        By definition,
        \begin{equation*}
            \Map_{\Shv(\deltaCat^n)}([1]^n ,\square^*\cubicalNerve X) \;\simeq\; \Map_{\Shv(\globe)}(\laxcube{n},X).
        \end{equation*} 
        
        The resulting natural map
        \begin{equation*}
            \Map_{\Shv(\globe)}(\laxdisk{n}, X) \;\too\; \Map_{\Shv(\globe)}(\laxcube{n}, X)
        \end{equation*}
        is given by the composition along the top and right arrows, followed by the inverse of the bottom arrow, in the following square:
        \begin{equation*}
            \begin{tikzcd}
                { \Map_{\Shv(\globe)}(\laxdisk{n}, X)} & {\Map_{\Shv(\deltaCat^n)}(\G^*\laxdisk{n},\square^*\cubicalNerve X)} \\
                { \Map_{\Shv(\globe)}(\laxcube{n}, X)} & {\Map_{\Shv(\deltaCat^n)}([1]^n,\square^*\cubicalNerve X).}
                \arrow[from=1-1, to=1-2]
                \arrow["{-\circ \pi_{n}}"', from=1-1, to=2-1]
                \arrow[from=1-2, to=2-2]
                \arrow["\sim", from=2-1, to=2-2]
            \end{tikzcd}
        \end{equation*}
        This square commutes, as by the definition of $\alpha$, both sides send a map $\laxdisk{n}\to X$ to the $[1]^n$-cell of $\square^*\cubicalNerve X$ given by 
        \begin{equation*}
            \laxcube{n} \;\xtoo{\pi_{n}}\; \laxdisk{n}\;\too\; X.
        \end{equation*}
        Therefore, the corresponding map $\laxcube{n} \to \laxdisk{n}$ is identified with $\pi_{n}$.

        \paragraph{(3) $\iota\partial[1]^n$.}
        Write
        \begin{equation*}
            \iota \partial [1]^n \;\simeq\; \colim_{[1]^I \in (\deltaCat^n)^\el_{/[1]^n},\ |I| < n} [1]^I,\qquad
            \iota \partial \laxcube{n} \;\simeq\; \colim_{\laxcube{k}\in (\grid)^\el_{/\laxcube{n}},\ k<n} \laxcube{k},
        \end{equation*}
        where the second isomorphism uses~\cref{lem:boundary-of-grid-in-fold}.
        There is an isomorphism between the two indexing categories given by $[1]^I\mapsto \laxcube{|I|}$ (\cref{lem:elementary-in-grid-and-delta}), so we get 
        \begin{equation*}
            \Map_{\Shv(\deltaCat^n)}(\iota\partial[1]^n ,\square^*\cubicalNerve X) \;\simeq\; \Map_{\Shv(\globe)}(\iota\partial\laxcube{n},X).
        \end{equation*} 

        The resulting natural map
        \begin{equation*}
            \Map_{\Shv(\globe)}(\laxcube{n}, X) \;\too\; \Map_{\Shv(\globe)}(\iota \partial \laxcube{n}, X)
        \end{equation*}
        corresponds to the map
        \begin{equation*}
            \iota\partial \laxcube{n} \;\simeq\; \colim_{\laxcube{k}\in (\grid)^\el_{/\laxcube{n}},\ k<n} \laxcube{k} \;\too\; \laxcube{n},
        \end{equation*}
        which is identified with the counit map.
        
        Combining the above, we obtain a natural map
        \begin{equation*}
            \Map_{\Shv(\globe)}(\laxdisk{n}, X) \;\too\; \Map_{\Shv(\globe)}(\iota \partial \laxdisk {n}, X) \times_{\Map_{\Shv(\globe)}(\iota \partial \laxcube{n}, X)} \Map_{\Shv(\globe)}(\laxcube{n},X)
        \end{equation*}
        which corresponds to the commutative square
        \begin{equation*}
            \begin{tikzcd}
                {\iota\partial\laxcube{n}} & {\iota\partial\laxdisk{n}} \\
                {\laxcube{n}} & {\laxdisk{n}}
                \arrow[from=1-1, to=1-2]
                \arrow[from=1-1, to=2-1]
                \arrow[from=1-2, to=2-2]
                \arrow["\pi_n", from=2-1, to=2-2]
            \end{tikzcd}
        \end{equation*}
        where the vertical maps are the counit maps of the adjunction $\iota \dashv \partial$. It follows that the top horizontal map is $\iota\partial \pi_n$. 
        Finally, this square is a pushout in $\Shv(\globe)$ by \cref{lem:pushout-fold-map}.
    \end{proof}

    \begin{remark}
    \label{rem:yoneda-and-inclusion-of-grids}
        For $Q\in \grid$, we can either consider $Q \in \Shv(\grid)$ by the Yoneda embedding or $Q\in \Shv(\globe)$ by the inclusion $\grid\subseteq \Shv(\globe)$. 
        These two perspectives are related by $\cubicalNerve$ and $\Fold$:
        \begin{equation*}
            \cubicalNerve Q \;\simeq\; Q \qin \Shv(\grid),
        \end{equation*}
        as the restricted Yoneda is simply the Yoneda embedding when restricted to $\grid$, and
        \begin{equation*}
            Q \;\simeq\; \Fold Q \qin \Shv(\globe),
        \end{equation*}
        by applying $\Fold$ to the above isomorphism and using~\cref{prop:unit-is-iso}.
    \end{remark}

\subsection{Reconstructing cubes}        
        
    In this subsection, we produce a pushout square 
    \begin{equation*}
        \begin{tikzcd}
            {\Fold^\L\iota\laxdiskb{n}} & {\iota\partial\laxcube{n}} \\
            {\Fold^\L\laxdisk{n}} & {\laxcube{n}}
            \arrow[from=1-1, to=1-2]
            \arrow[from=1-1, to=2-1]
            \arrow[from=1-2, to=2-2]
            \arrow[from=2-1, to=2-2]
            \arrow["\lrcorner"{anchor=center, pos=0.125, rotate=180}, draw=none, from=2-2, to=1-1]
        \end{tikzcd}
        \qin \Shv(\grid),
    \end{equation*}
    which is a $\grid$-analog of~\cref{cor:pushout-square-and-globe}.
    
    Consider the map $[1]^n \to \G^*\laxdisk{n}$ selecting the identity cell $\laxdisk{n}\to \laxdisk{n}$.
    Applying $\square_!$, we get a map
    \begin{equation*}
        \laxcube{n} \;\too\; \Fold^\L \laxdisk{n} \qin \Shv(\grid).
    \end{equation*}
    In the other direction, consider the composition
    \begin{equation*}
        \laxdisk{n} \;\too\; \laxcube{n}\xtoo{\alpha} \Fold \cubicalNerve \laxcube{n} \;\simeq\; \Fold \laxcube{n},
    \end{equation*}
    where the first map selects the top-dimensional cell and the last isomorphism keeps the Yoneda embedding implicit (see~\cref{rem:yoneda-and-inclusion-of-grids}).
    This corresponds by adjunction to
    \begin{equation*}
        \Fold^\L \laxdisk{n} \;\too\; \laxcube{n} \qin \Shv(\grid).
    \end{equation*}
    
    \begin{lemma}
    \label{lem:fold-map-in-grids}
        The composition of the above maps
        \begin{equation*}
            \laxcube{n} \;\too\; \Fold^\L \laxdisk{n} \;\too\; \laxcube{n}\qin \Shv(\grid)
        \end{equation*}
        is equal to the fold map $\Phi_{n}\colon \laxcube{n}\to\laxcube{n}$.
    \end{lemma}

    \begin{proof}
        The map  $\laxdisk{n}\to \Fold\laxcube{n}$ corresponds by adjunction to  $\G^*\laxdisk{n}\to \square^*\laxcube{n}$, which, by the definition of $\alpha$, sends a globular cell $\G[\vec{a}]\to \laxdisk{n}$ to the cubical cell
        \begin{equation*}
            \square[\vec{a}] \;\xtoo{\pi_{[\vec{a}]}}\;\G[\vec{a}] \;\too\; \laxdisk{n} \;\too\; \laxcube{n}.
        \end{equation*}
        In particular, as $[1]^n \to \G^*\laxdisk{n}$ picks the identity cell, the composition $[1]^n\to \G^*\laxdisk{n}\to \square^*\laxcube{n}$
        picks the cubical cell
        \begin{equation*}
            \laxcube{n} \;\xtoo{\pi_{n}}\;\laxdisk{n} \;\too\; \laxcube{n},
        \end{equation*}
        which is equal to $\Phi_{n}$ by~\cref{lem:endomorphism-equal-fold-map}.
    \end{proof}

    We now consider the same maps on the boundary.
    As $\G^*$ commutes with $\partial$, we get a map 
    \begin{equation*}
        \partial [1]^n \;\too\; \partial \G^*\laxdisk{n} \;\simeq\; \G^*\partial \laxdisk{n}.
    \end{equation*}
    Applying $\square_!$, and using~\cref{lem:square-bc-on-representables}, we get a map 
    \begin{equation*}
        \partial\laxcube{n} \;\too\; \Fold^\L \partial\laxdisk{n} \qin \Shv(\grid).
    \end{equation*}
    Similarly, as $\Fold$ commutes with $\partial$, we get a map 
    \begin{equation*}
        \partial \laxdisk{n} \;\too\; \partial \Fold \laxcube{n} \;\simeq\; \Fold\partial \laxcube{n},
    \end{equation*}
    which corresponds by adjunction to
    \begin{equation*}
        \Fold^\L \partial \laxdisk{n} \;\too\; \partial \laxcube{n} \qin \Shv(\grid).
    \end{equation*}
    
    \begin{lemma}
    \label{cor:boundary-fold-map-in-grids}
        The composition of the above maps
        \begin{equation*}
            \partial \laxcube{n} \;\too\; \Fold^\L \partial \laxdisk{n} \;\too\; \partial \laxcube{n} \qin \Shv(\grid[n-1])
        \end{equation*}
        is equal to the boundary of the fold map $\partial\Phi_{n}\colon \partial\laxcube{n}\to\partial\laxcube{n}$.
    \end{lemma} 

    \begin{proof}
        Consider the Beck--Chevalley comparison map $\beta \colon \Fold^\L\partial \To \partial\Fold^\L$, which we do not yet know is invertible. The map $\Fold^\L \partial\laxdisk{n}\to \partial\laxcube{n}$ factors through it
        \begin{equation*}
            \Fold^\L \partial \laxdisk{n} \;\xtoo{\beta}\; \partial \Fold^\L \laxdisk{n} \;\too\; \partial \laxcube{n}
        \end{equation*}
        where the second map is the boundary of $\Fold^\L\laxdisk{n}\to \laxcube{n}$.
        We claim that the composition
        \begin{equation*}
            \partial \laxcube{n} \;\too\; \Fold^\L \partial \laxdisk{n} \;\xtoo{\beta}\; \partial \Fold^\L \laxdisk{n}
        \end{equation*}
        is equal to the boundary of $\laxcube{n}\to\Fold^\L\laxdisk{n}$, after which the result follows from~\cref{lem:fold-map-in-grids}.
        To verify the claim, we unwind the definitions and reduce to the composition
        \begin{equation*}
            \square_!\partial[1]^n 
            \;\too\; 
            \square_!\partial \G^*\laxdisk{n} 
            \;\too\;
            \partial\square_! \G^*\laxdisk{n} 
        \end{equation*}
        which uses the Beck--Chevalley comparison map $\square_!\partial \To \partial\square_!$.
        By naturality, we obtain a commuting square
        \begin{equation*}
            \begin{tikzcd}
                {\square_!\partial[1]^n} & {\square_!\partial\G^*\laxdisk{n}} & \\
                {\partial\square_![1]^n} & {\partial\square_!\G^*\laxdisk{n}} & {}
                \arrow[from=1-1, to=1-2]
                \arrow["{\rotatebox{90}{$\sim$}}"', from=1-1, to=2-1]
                \arrow[from=1-2, to=2-2]
                \arrow[from=2-1, to=2-2]
            \end{tikzcd}
        \end{equation*} 
        in which the left vertical map is an isomorphism by~\cref{lem:square-bc-on-representables}, and the bottom horizontal map is identified with the boundary of $\laxcube{n} \to \Fold^\L \laxdisk{n}$.
    \end{proof}
    Applying $\iota$ to the composition in~\cref{cor:boundary-fold-map-in-grids}, and using the fact that it commutes with $\Fold^\L$, we get that the composition
    \begin{equation*}
        \iota \partial \laxcube{n} \;\too\; \Fold^\L \iota \partial \laxdisk{n} \;\too\; \iota \partial \laxcube{n} \qin \Shv(\grid)
    \end{equation*}
    is equal to $\iota\partial\Phi_n$.
    Combining this with~\cref{lem:fold-map-in-grids}, we get a commutative diagram
    \begin{equation*}\begin{tikzcd}
        {\iota\partial\laxcube{n}} & {\Fold^\L\iota\laxdiskb{n}} & {\iota\partial\laxcube{n}} \\
        {\laxcube{n}} & {\Fold^\L\laxdisk{n}} & {\laxcube{n}.}
        \arrow[from=1-1, to=1-2]
        \arrow["{\iota\partial\Phi_{n}}", curve={height=-12pt}, from=1-1, to=1-3]
        \arrow[from=1-1, to=2-1]
        \arrow[from=1-2, to=1-3]
        \arrow[from=1-2, to=2-2]
        \arrow[from=1-3, to=2-3]
        \arrow[from=2-1, to=2-2]
        \arrow["{\Phi_{n}}"', curve={height=12pt}, from=2-1, to=2-3]
        \arrow[from=2-2, to=2-3]
    \end{tikzcd}\end{equation*}
    Our goal is to prove that the right square is a pushout square.
    We will do so by proving that the outer rectangle and the left square are pushouts. 
    For the outer rectangle, we will reduce to the basic fold map $\Phi_2=\psi$.

    \begin{lemma} 
    \label{lem:pushout-basic-fold-map}
        The following is a pushout square in $\Seg(\grid[2])$:
        \begin{equation*}\begin{tikzcd}
            {\iota\partial\laxcube{2}} & {\iota\partial\laxcube{2}} \\
            {\laxcube{2}} & {\laxcube{2}.}
            \arrow[""{name=0, anchor=center, inner sep=0}, "{\iota\partial\psi}", from=1-1, to=1-2]
            \arrow[from=1-1, to=2-1]
            \arrow[from=1-2, to=2-2]
            \arrow["{\psi}", from=2-1, to=2-2]
            \arrow["\lrcorner"{anchor=center, pos=0.125, rotate=180}, draw=none, from=2-2, to=0]
        \end{tikzcd}\end{equation*}
    \end{lemma}

    \begin{proof}
        Denote the inclusion of the boundary by $\delta\colon \iota\partial \laxcube{2}\to \laxcube{2}$.
        We need to prove that for every $X\in\Seg(\grid[2])$, the canonical map
        \begin{equation*}
            A\;\colon\;\Map_{\Seg(\grid[2])}(\laxcube{2}, X) \;\too\; \Map_{\Seg(\grid[2])}(\laxcube{2}, X) \times_{\Map_{\Seg(\grid[2])}(\iota \partial \laxcube{2}, X)} {\Map_{\Seg(\grid[2])}(\iota \partial \laxcube{2}, X)}
        \end{equation*}
        is an isomorphism.
        An element of the pullback on the right is given by a pair of maps
        \begin{equation*}
            x \;\colon\; \laxcube{2} \;\too\; X, \qquad
            y \;\colon\; \iota \partial \laxcube{2} \;\too\; X
        \end{equation*}
        together with an identification $\delta^*x\simeq (\iota\partial\psi)^*y$, which we keep implicit.
        The data of a map $\iota\partial \laxcube{2}\to X$ is a non-commuting square in $X$, i.e.\ a choice of four arrows with compatible sources and targets.
        Thus, we can denote the faces composing $y$ and $x$ as follows:
        \begin{equation*}
        y \;=\;
        \begin{tikzcd}
            a & b \\
            c & d
            \arrow["f", from=1-1, to=1-2]
            \arrow["h"', from=1-1, to=2-1]
            \arrow["g", from=1-2, to=2-2]
            \arrow["k"', from=2-1, to=2-2]
        \end{tikzcd}
        \hspace{30pt}
        x \;=\;
        \begin{tikzcd}[sep=2.25em]
            a & d \\
            a & d.
            \arrow["gf", from=1-1, to=1-2]
            \arrow[equals, from=1-1, to=2-1]
            \arrow["x"{description}, Rightarrow, from=1-2, to=2-1]
            \arrow[equals, from=1-2, to=2-2]
            \arrow["kh"', from=2-1, to=2-2]
        \end{tikzcd}\end{equation*}
        
        The above map sends $z\colon \laxcube{2}\to X$ to the pair $A(z)=(\psi^*z,\delta^*z)$, and $\psi^*z$ is given by composing connections to the left and right of $z$:
        \begin{equation*}
            \begin{tikzcd}[sep=2.25em]
                a & b && a & a & b & d \\
                c & d && a & c & d & {d.}
                \arrow["f", from=1-1, to=1-2]
                \arrow["h"', from=1-1, to=2-1]
                \arrow["z"{description}, Rightarrow, from=1-2, to=2-1]
                \arrow[""{name=0, anchor=center, inner sep=0}, "g", from=1-2, to=2-2]
                \arrow[equals, from=1-4, to=1-5]
                \arrow[""{name=1, anchor=center, inner sep=0}, equals, from=1-4, to=2-4]
                \arrow["f", from=1-5, to=1-6]
                \arrow["{({\gamma^{+})^*h}}"{description}, Rightarrow, from=1-5, to=2-4]
                \arrow["h"{description}, from=1-5, to=2-5]
                \arrow["g", from=1-6, to=1-7]
                \arrow["z"{description}, Rightarrow, from=1-6, to=2-5]
                \arrow["g"{description}, from=1-6, to=2-6]
                \arrow["{({\gamma^{-})^*g}}"{description}, Rightarrow, from=1-7, to=2-6]
                \arrow[equals, from=1-7, to=2-7]
                \arrow["k"', from=2-1, to=2-2]
                \arrow["h"', from=2-4, to=2-5]
                \arrow["k"', from=2-5, to=2-6]
                \arrow[equals, from=2-6, to=2-7]
                \arrow[between={0.4}{0.6}, maps to, from=0, to=1]
            \end{tikzcd}
        \end{equation*}
        We construct a map in the opposite direction
        \begin{equation*}
            B\colon \Map_{\Seg(\grid[2])}(\laxcube{2},X) \times_{\Map_{\Seg(\grid[2])}(\iota\partial\laxcube{2},X)}{\Map_{\Seg(\grid[2])}(\iota\partial\laxcube{2},X)}
            \;\too\; \Map_{\Seg(\grid[2])}(\laxcube{2},X)
        \end{equation*}
        by sending $(x,y)$ as above to the following composition:
        \begin{equation*}\begin{tikzcd}[sep=2.25em]
            a & c & c \\
            a & c & d \\
            a & c & d \\
            c & c & d.
            \arrow["f", from=1-1, to=1-2]
            \arrow[equals, from=1-1, to=2-1]
            \arrow[equals, from=1-2, to=1-3]
            \arrow["{\sigma_2^*f}"{description}, Rightarrow, from=1-2, to=2-1]
            \arrow[equals, from=1-2, to=2-2]
            \arrow["{(\gamma^{+})^*g}"{description}, Rightarrow, from=1-3, to=2-2]
            \arrow["g", from=1-3, to=2-3]
            \arrow["f"{description}, from=2-1, to=2-2]
            \arrow[equals, from=2-1, to=3-1]
            \arrow[""{name=0, anchor=center, inner sep=0}, "g"{description}, from=2-2, to=2-3]
            \arrow[equals, from=2-3, to=3-3]
            \arrow[""{name=1, anchor=center, inner sep=0}, "h"{description}, from=3-1, to=3-2]
            \arrow["h"', from=3-1, to=4-1]
            \arrow["k"{description}, from=3-2, to=3-3]
            \arrow["{(\gamma^{-})^*h}"{description}, Rightarrow, from=3-2, to=4-1]
            \arrow[equals, from=3-2, to=4-2]
            \arrow["{\sigma_2^*k}"{description}, Rightarrow, from=3-3, to=4-2]
            \arrow[equals, from=3-3, to=4-3]
            \arrow[equals, from=4-1, to=4-2]
            \arrow["k"', from=4-2, to=4-3]
            \arrow["x"{description}, between={0.1}{0.9}, Rightarrow, from=0, to=1]
        \end{tikzcd}\end{equation*}
        We claim that $A$ and $B$ are inverse to each other.
        Let $z\colon \laxcube{2}\to X$.
        Then $BA(z)$ is the map $\laxcube{2}\to X$ given by the following composition:
        \begin{equation*}
            \begin{tikzcd}[sep=2.25em]
                a & a & b & b \\
                a & a & b & d \\
                a & c & d & d \\
                c & c & d & {d.}
                \arrow[equals, from=1-1, to=1-2]
                \arrow[equals, from=1-1, to=2-1]
                \arrow["f", from=1-2, to=1-3]
                \arrow["{\sigma_{\{1,2\}}^*a}"{description}, Rightarrow, from=1-2, to=2-1]
                \arrow[equals, from=1-2, to=2-2]
                \arrow[equals, from=1-3, to=1-4]
                \arrow["{{\sigma_2^*f}}"{description}, Rightarrow, from=1-3, to=2-2]
                \arrow[equals, from=1-3, to=2-3]
                \arrow["{{({\gamma^{+})^*g}}}"{description}, Rightarrow, from=1-4, to=2-3]
                \arrow["g", from=1-4, to=2-4]
                \arrow[equals, from=2-1, to=2-2]
                \arrow[equals, from=2-1, to=3-1]
                \arrow["f"{description}, from=2-2, to=2-3]
                \arrow["{{({\gamma^{+})^*h}}}"{description}, Rightarrow, from=2-2, to=3-1]
                \arrow["h"{description}, from=2-2, to=3-2]
                \arrow["g"{description}, from=2-3, to=2-4]
                \arrow["z"{description}, Rightarrow, from=2-3, to=3-2]
                \arrow["g"{description}, from=2-3, to=3-3]
                \arrow["{{({\gamma^{-})^*g}}}"{description}, Rightarrow, from=2-4, to=3-3]
                \arrow[equals, from=2-4, to=3-4]
                \arrow["h"{description}, from=3-1, to=3-2]
                \arrow["h"', from=3-1, to=4-1]
                \arrow["k"{description}, from=3-2, to=3-3]
                \arrow["{{({\gamma^{-})^*h}}}"{description}, Rightarrow, from=3-2, to=4-1]
                \arrow[equals, from=3-2, to=4-2]
                \arrow[equals, from=3-3, to=3-4]
                \arrow["{{\sigma_2^*k}}"{description}, Rightarrow, from=3-3, to=4-2]
                \arrow[equals, from=3-3, to=4-3]
                \arrow["{\sigma_{\{1,2\}}^*d}"{description}, Rightarrow, from=3-4, to=4-3]
                \arrow[equals, from=3-4, to=4-4]
                \arrow[equals, from=4-1, to=4-2]
                \arrow["k"', from=4-2, to=4-3]
                \arrow[equals, from=4-3, to=4-4]
            \end{tikzcd}
        \end{equation*}
        The vertical compositions of connections yield horizontal degeneracies $\sigma_1^* h$ and $\sigma_1^*g$, respectively.
        It follows that the above composition is naturally isomorphic to $z$.
        
        The verification that $AB(x,y)\simeq (x,y)$ is similar, using instead that horizontal compositions of connections produce vertical degeneracies.
    \end{proof}

    \begin{lemma} 
    \label{lem:pushout-fold-map}
        For every $1\le i\le n-1$, the following is a pushout square in $\Seg(\grid)$:
        \begin{equation*}
            \begin{tikzcd}
                {\iota\partial\laxcube{n}} & {\iota\partial\laxcube{n}} \\
                {\laxcube{n}} & {\laxcube{n}.}
                \arrow[""{name=0, anchor=center, inner sep=0}, "{\iota\partial\psi_i}", from=1-1, to=1-2]
                \arrow[from=1-1, to=2-1]
                \arrow[from=1-2, to=2-2]
                \arrow["{\psi_i}", from=2-1, to=2-2]
                \arrow["\lrcorner"{anchor=center, pos=0.125, rotate=180}, draw=none, from=2-2, to=0]
            \end{tikzcd}
        \end{equation*}
    \end{lemma}
    
    \begin{proof}
        We will show the case $i=1$, the rest are similar.
        Consider the following pushout square in $\Seg(\deltaCat^n)$ by~\cref{cor:boundary-of-external-product}, and its image in $\Seg(\grid)$ under $\square_!$ by~\cref{lem:square-bc-on-representables,exm:gray-is-strict-on-rep}:
        \begin{equation*}
            \begin{tikzcd}
                {\iota\partial[1]^2\boxtimes\iota\partial[1]^{n-2}} & {[1]^2\boxtimes\iota\partial[1]^{n-2}} && {\iota\partial\laxcube{2}\xlax\iota\partial\laxcube{n-2}} & {\laxcube{2}\xlax\iota\partial\laxcube{n-2}} \\
                {\iota\partial[1]^2\boxtimes[1]^{n-2}} & {\iota \partial[1]^{n}} && {\iota\partial \laxcube{2}\xlax\laxcube{n-2}} & {\iota\partial\laxcube{n}.}
                \arrow[""{name=0, anchor=center, inner sep=0}, from=1-1, to=1-2]
                \arrow[from=1-1, to=2-1]
                \arrow[""{name=1, anchor=center, inner sep=0}, from=1-2, to=2-2]
                \arrow[""{name=2, anchor=center, inner sep=0}, from=1-4, to=1-5]
                \arrow[""{name=3, anchor=center, inner sep=0}, from=1-4, to=2-4]
                \arrow[from=1-5, to=2-5]
                \arrow[from=2-1, to=2-2]
                \arrow[from=2-4, to=2-5]
                \arrow["{\square_!}", between={0.4}{0.6}, maps to, from=1, to=3]
                \arrow["\lrcorner"{anchor=center, pos=0.125, rotate=180}, draw=none, from=2-2, to=0]
                \arrow["\lrcorner"{anchor=center, pos=0.125, rotate=180}, draw=none, from=2-5, to=2]
            \end{tikzcd}
        \end{equation*}
        Recall that $\psi_1=\psi\xlaxs\id_{\laxcube{n-2}}=\psi\xlax\id_{\laxcube{n-2}}$.
        It suffices to prove that the required square is a pushout square after replacing $\iota\partial\psi_1$ with each of its restrictions
        \begin{equation*}
            \begin{tikzcd}
                {\iota\partial\laxcube{2}\xlax\iota\partial\laxcube{n-2}} & {\iota\partial\laxcube{2}\xlax\iota\partial\laxcube{n-2}} \\
                {\laxcube{n}} & {\laxcube{n},}
                \arrow[""{name=0, anchor=center, inner sep=0}, "{\iota\partial \psi\xlax \id}", from=1-1, to=1-2]
                \arrow[from=1-1, to=2-1]
                \arrow[from=1-2, to=2-2]
                \arrow["{\psi \xlax \id}", from=2-1, to=2-2]
                \arrow["\lrcorner"{anchor=center, pos=0.125, rotate=180}, draw=none, from=2-2, to=0]
            \end{tikzcd}
            \qquad
            \begin{tikzcd}
                {\iota\partial \laxcube{2}\xlax\laxcube{n-2}} & {\iota\partial \laxcube{2}\xlax\laxcube{n-2}} \\
                {\laxcube{n}} & {\laxcube{n},}
                \arrow[""{name=0, anchor=center, inner sep=0}, "{\iota\partial\psi \xlax \id}", from=1-1, to=1-2]
                \arrow[from=1-1, to=2-1]
                \arrow[from=1-2, to=2-2]
                \arrow["{\psi \xlax \id}", from=2-1, to=2-2]
                \arrow["\lrcorner"{anchor=center, pos=0.125, rotate=180}, draw=none, from=2-2, to=0]
            \end{tikzcd} 
        \end{equation*}
        \begin{equation*}
            \begin{tikzcd}
                {\laxcube{2}\xlax\iota\partial\laxcube{n-2}} & {\laxcube{2}\xlax\iota\partial\laxcube{n-2}} \\
                {\laxcube{n}} & {\laxcube{n}.}
                \arrow[""{name=0, anchor=center, inner sep=0}, "{\psi \xlax \id}", from=1-1, to=1-2]
                \arrow[from=1-1, to=2-1]
                \arrow[from=1-2, to=2-2]
                \arrow["{\psi \xlax \id}", from=2-1, to=2-2]
                \arrow["\lrcorner"{anchor=center, pos=0.125, rotate=180}, draw=none, from=2-2, to=0]
            \end{tikzcd}
        \end{equation*}
        The first two are pushout squares by~\cref{lem:pushout-basic-fold-map}, using the fact that $\xlax$ commutes with colimits in each variable.    
        The last square is the Gray product of two degenerate squares.
    \end{proof}

    \begin{proposition}
    \label{prop:reconstructing-cubes}
        The following is a pushout square in $\Seg(\grid)$:
        \begin{equation*}\begin{tikzcd}
            {\Fold^\L\iota\laxdiskb{n}} & {\iota\partial\laxcube{n}} \\
            {\Fold^\L\laxdisk{n}} & {\laxcube{n}.}
            \arrow[from=1-1, to=1-2]
            \arrow[from=1-1, to=2-1]
            \arrow[from=1-2, to=2-2]
            \arrow[from=2-1, to=2-2]
            \arrow["\lrcorner"{anchor=center, pos=0.125, rotate=180}, draw=none, from=2-2, to=1-1]
        \end{tikzcd}\end{equation*}
    \end{proposition}
    
    \begin{proof}
        Recall the commutative diagram from~\cref{lem:fold-map-in-grids,cor:boundary-fold-map-in-grids}:
        \begin{equation*}\begin{tikzcd}
            {\iota\partial\laxcube{n}} & {\Fold^\L\iota\laxdiskb{n}} & {\iota\partial\laxcube{n}} \\
            {\laxcube{n}} & {\Fold^\L\laxdisk{n}} & {\laxcube{n}.}
            \arrow[from=1-1, to=1-2]
            \arrow["{\iota\partial\Phi_{n}}", curve={height=-12pt}, from=1-1, to=1-3]
            \arrow[from=1-1, to=2-1]
            \arrow[from=1-2, to=1-3]
            \arrow[from=1-2, to=2-2]
            \arrow[from=1-3, to=2-3]
            \arrow[from=2-1, to=2-2]
            \arrow["{\Phi_{n}}"', curve={height=12pt}, from=2-1, to=2-3]
            \arrow[from=2-2, to=2-3]
        \end{tikzcd}\end{equation*}
        To prove that the right square is a pushout, it suffices to prove that the left square and the outer rectangle are pushouts.
        For the outer rectangle it follows from~\cref{lem:pushout-fold-map}, where we recall that
        \begin{equation*}
            \Phi_{n} \;=\; \psi_{n-1} (\psi_{n-2} \psi_{n-1}) \cdots (\psi_1 \psi_2 \cdots \psi_{n-1}).
        \end{equation*}
        The left square is obtained by applying $\square_!$ to the square 
        \begin{equation*}\begin{tikzcd}
            {\iota\partial[1]^{n}} & {\G^*\iota\laxdiskb{n}} \\
            {[1]^{n}} & {\G^*\laxdisk{n}}
            \arrow[""{name=0, anchor=center, inner sep=0}, from=1-1, to=1-2]
            \arrow[from=1-1, to=2-1]
            \arrow[from=1-2, to=2-2]
            \arrow[from=2-1, to=2-2]
            \arrow["\lrcorner"{anchor=center, pos=0.125, rotate=180}, draw=none, from=2-2, to=0]
        \end{tikzcd}
        \hspace{10pt}
        \qin \Seg(\deltaCat^{n})\end{equation*}
        which is a pushout square by~\cref{lem:n-uple-globe-pushout}.
    \end{proof}

\subsection{The main theorem}
    
    For $Y\in \Seg(\grid)$ and $Q\in \grid$, consider the composition
    \begin{equation*}
        \Map_{\Shv(\grid)}(Q, Y) \;\too\; \Map_{\Shv(\globe)}(\Fold Q, \Fold Y) \;\simeq\; \Map_{\Shv(\grid)}(Q, \cubicalNerve \Fold Y),
    \end{equation*}
    where the last isomorphism is by the definition of $\cubicalNerve$, using the notation from~\cref{rem:yoneda-and-inclusion-of-grids}.
    This corresponds to a natural transformation 
    \begin{equation*}
        \beta \;\colon\; \id_{\Seg(\grid)} \;\Too\; \cubicalNerve\Fold.
    \end{equation*}

    \begin{proposition}
    \label{prop:counit-is-iso}
        The above natural transformation is an isomorphism
        \begin{equation*}
           \beta \;\colon\; \id_{\Seg(\grid)} \;\isoToo\; \cubicalNerve\Fold.
        \end{equation*}
    \end{proposition}

    \cref{prop:unit-is-iso,prop:counit-is-iso} immediately imply
    
    \begin{theorem}
    \label{thm:equivalence-n}
        The functors
        \begin{equation*}
            \Fold \;\colon\; \Shv(\grid) \;\longadj\; \Shv(\globe ) \;\cocolon\; \cubicalNerve.
        \end{equation*}
        are inverse equivalences of categories.
    \end{theorem}
    
    \begin{proof}[Proof of~\cref{prop:counit-is-iso}]
        We prove the claim by induction, in particular allowing us to use \cref{thm:equivalence-n} under the induction hypothesis. 
        The case $n=0$ is trivial.
        Let $Y\in \Shv(\grid)$.
        It is enough to check that $\beta_Y\colon Y\to \cubicalNerve\Fold Y$ is an isomorphism after mapping from the weak generator $\laxcube{n}$.
        By the definition of $\beta$, we need to show that 
        \begin{equation*}
            \Map_{\Shv(\grid)}(\laxcube{n},Y) \;\too\; \Map_{\Shv(\globe)}(\Fold \laxcube{n},\Fold Y)
        \end{equation*}
        is an isomorphism, or equivalently, that the counit $\Fold^\L\Fold\laxcube{n}\to \laxcube{n}$ is an isomorphism.

        By induction, and by~\cref{thm:equivalence-n}, the counit $\Fold^\L\Fold\partial\laxcube{n}\to \partial \laxcube{n}$
        is an isomorphism.
        Thus, there are isomorphisms
        \begin{equation*}
            \Fold^\L \iota \partial \Fold \laxcube{n} \;\simeq\; \iota \Fold^\L \Fold \partial \laxcube{n} \;\isotoo\; \iota \partial \laxcube{n},
        \end{equation*}
        where we use the fact that $\Fold$ commutes with $\partial$ and $\Fold^\L$ commutes with $\iota$.
        The above composition fits into a commutative cube
        \begin{equation*}
            \begin{tikzcd}
                {\Fold^\L\iota\laxdiskb{n}} && {\Fold^\L\iota\partial\Fold\laxcube{n}} & \\
                & {\Fold^\L\iota\laxdiskb{n}} && {\iota\partial\laxcube{n}} \\
                {\Fold^\L\laxdisk{n}} && {\Fold^\L\Fold\laxcube{n}} \\
                \\
                & {\Fold^\L\laxdisk{n}} && {\laxcube{n}}
                \arrow[""{name=0, anchor=center, inner sep=0}, from=1-1, to=1-3]
                \arrow[equals, from=1-1, to=2-2]
                \arrow[from=1-1, to=3-1]
                \arrow["\begin{array}{c} \substack{\sim \\ \\} \end{array}"{marking, allow upside down}, from=1-3, to=2-4]
                \arrow[from=1-3, to=3-3]
                \arrow[from=2-4, to=5-4]
                \arrow[from=3-1, to=3-3]
                \arrow[equals, from=3-1, to=5-2]
                \arrow[from=3-3, to=5-4]
                \arrow[from=5-2, to=5-4]
                \arrow[""{name=1, anchor=center, inner sep=0}, from=2-2, to=2-4, crossing over]
                \arrow[from=2-2, to=5-2, crossing over]
                \arrow["\lrcorner"{anchor=center, pos=0.125, rotate=180}, draw=none, from=3-3, to=0]
                \arrow["\lrcorner"{anchor=center, pos=0.125, rotate=180}, draw=none, from=5-4, to=1]
            \end{tikzcd}
            \qin \Shv(\grid),
        \end{equation*}
        where the front face is the pushout square of~\cref{prop:reconstructing-cubes}, and the back face is the result of applying $\Fold^\L$ to the pushout square of~\cref{cor:pushout-square-and-globe}.
        It follows that the map between the pushouts, which is the counit $\Fold^\L\Fold \laxcube{n}\to \laxcube{n}$, is also an isomorphism.
    \end{proof}

    As both $\Fold$ and $\cubicalNerve$ commute with $\partial$, they induce functors on the limits
    \begin{equation*}
        \Fold \;\colon\; \Shv(\grid[]) \;\longadj\; \Shv(\globe[]) \;\cocolon\; \cubicalNerve.
    \end{equation*}
    Both sides are also equipped with a monoidal structure: $\Shv(\grid[])$ with the one from~\cref{cor:gray-product}, and $\Shv(\globe[])$ with the Gray product of flagged $\infty$-categories constructed in~\cite{Campion-2023-Gray}.
    
    \begin{theorem}
    \label{thm:equivalence-infty}
        The functors
        \begin{equation*}
            \Fold \;\colon\; \Shv(\grid[]) \;\longadj\; \Shv(\globe[]) \;\cocolon\; \cubicalNerve.
        \end{equation*}
        are inverse equivalences of monoidal categories.
    \end{theorem}

    \begin{proof}
        The fact that they are inverse equivalences of categories follows from~\cref{thm:equivalence-n}.
        The fact that the equivalence is monoidal follows from the uniqueness of the Gray product~\cite[Theorem~4.1, Example~4.3]{Campion-2023-Gray}:
        it is the unique presentably monoidal structure on flagged $\infty$-categories that agrees with the strict Gray product on the full subcategory generated by $\laxcube{n}$, $n<\infty$.
    \end{proof}

    Finally, we show that this equivalence restricts to univalent sheaves, thereby exhibiting $\Shv(\grid)\univ$ as a model for $n$-categories.

    \begin{corollary}
    \label{cor:equivalence-univalence}
        The functors $\Fold$ and $\cubicalNerve$ preserve univalent sheaves, and hence induce equivalences of categories
        \begin{equation*}
            \Fold \;\colon\; \Shv(\grid)\univ \;\longadj\; \Shv(\globe )\univ \;\cocolon\; \cubicalNerve,
        \end{equation*}
        \begin{equation*}
            \Fold \;\colon\; \Shv(\grid[])\univ \;\longadj\; \Shv(\globe[] )\univ \;\cocolon\; \cubicalNerve.
        \end{equation*}
    \end{corollary}

    \begin{proof}
        Suppose that $Y\in \Shv(\grid)\univ$.
        By definition, $\square^* Y\in \Shv(\deltaCat^n)$ is univalent, so $\Fold Y=\G_*\square^* Y\in \Shv(\globe)$ is univalent by \cref{cor:univalent-fold}. 
        In the other direction, suppose that $X\in \Shv(\globe)\univ$.
        We want to show that the following Segal space is univalent:
        \begin{equation*}
            (\square^*\cubicalNerve X)_{a_1,\dots,a_{k-1},\bullet,b_1,\dots,b_{n-k}} \;\simeq\; \Map_{\Shv(\globe)}(\square[\vec{a}]\xlaxs[\bullet]\xlaxs\square[\vec{b}],X)\qin \Shv(\deltaCat).
        \end{equation*}
        An arrow in this Segal space is a lax grid in $X$ of width $1$ in the $k$-direction, and composition is given by gluing in this direction.
        If such a lax grid is invertible, then all of its cells in the $k$-direction are invertible.
        Since $X$ is univalent, this implies that an invertible lax grid is isomorphic to a degenerate one, as required.
    \end{proof}

\bibliographystyle{alpha}
\phantomsection\addcontentsline{toc}{section}{\refname}
\bibliography{references}

\end{document}